\documentclass[final,3p,times,12pt]{elsarticle}

\usepackage{epsfig}
\usepackage{amssymb}
\usepackage{amsthm}
\usepackage{lineno}
\usepackage{amsmath}
\usepackage{graphicx}
\usepackage{subfigure}
\usepackage{color}
\usepackage[all]{xy}
\usepackage{booktabs}
\usepackage{tabularx}
\usepackage{float}
\usepackage{fancyhdr}
\usepackage{mathrsfs}
\usepackage{extarrows}

\theoremstyle{plain}
\newtheorem{theorem}{Theorem}[section]

\theoremstyle{definition}
\newtheorem{definition}[theorem]{Definition}
\newtheorem{example}[theorem]{Example}

\theoremstyle{remark}
\newtheorem{remark}{Remark}

\numberwithin{equation}{section}

\allowdisplaybreaks[4]

\begin{document}

\begin{frontmatter}
\title{
Betweenness relations and fuzzy betweenness relations in KM-fuzzy metric spaces
\tnoteref{label1}}
\tnotetext[label1]{This research is supported by 
National Natural Science Foundation of China (No. 12461091),
North China University of Technology organized scientific research
project (No. 2023YZZKY19,~No.2024NCUTYXCX104)}
\author[1]{Yu Zhong\corref{cor1}}
\cortext[cor1]{Corresponding author:~zhongyu@ncut.edu.cn}
\address[1]{College of Science, North China University of Technology, Beijing, China}

\begin{abstract}
In this paper, we mainly discuss 
the constructions and the characteristics of
betweenness relations and fuzzy betweenness relations in KM-fuzzy metric spaces. 
And the family of betweenness relations induced by a KM-fuzzy metric form a nest of betweenness relations.
The main focus of this paper is to introduce two different construction methods for fuzzy betweenness relations induced by a KM-fuzzy metric.
One of them is directly obtained by using the implication operator. The other is through the corresponding nest of metrics of KM-fuzzy metrics.
Furthermore, we also show that the two types of fuzzy betweenness relations are the same, and they also satisfy the eight kinds of four-point transitivity properties and the six kinds of five-point transitivity properties.  
\end{abstract}

\end{frontmatter}


\section{\bf Introduction}

The betweenness relation first emerged in geometry, which describes the situation when an element is in between two other elements. Its origin can be traced back to Pash's paper \cite{Pasch-1882} on distance geometry in 1882.
Subsequently, Huntington and Kline \cite{Huntington-1917,Huntington-1924} systematically discussed the independent postulates for betweenness relations and presented eight
kinds of four-point transitivity properties of a ternary relation. At the same time, a systematic and comprehensive analysis was provided to reveal the connections and counterexamples among these transitive relationships.
Based on this, Pitcher and Smiley \cite{Pitcher-1942} 
continued the analysis of Huntington and Kline, and then conducted a study on the transitivities of betweenness involving five points. 
Recently, L. Zedam, N. Bakri and B. De Bates \cite{Zedam-2018}
gave six kinds of five-point transitivity properties from the perspective of compositions of ternary relations and discuss the traces, the closures and openings of ternary relations.
All of these held a central position in the research of the betweenness theory.

The current research on the concept of betweenness 
relations is conducted within the framework of an axiomatic structure.
A betweenness relation is a special kind of ternary relation that satisfies reflexivity, symmetry, antisymmetry, and transitivity, which is widely associated with various mathematical structures, such as posets \cite{Rautenbach-2011,Sholander-1952, Zhang2019},
lattices \cite{Duvelmeyer-2004, Sholander-1952},
metrics \cite{Simovici-2009, Smiley-1943},
median algebras \cite{Chajda-2013, HedlÍková-1983}, 
normed spaces \cite{Diminnie-1981,Simovici-2009,Toranzos-1971},
convex structures \cite{Chvátal-2009, Vel-1993},
road systems \cite{Bankston-2013}, 
and data aggregations \cite{Pérez-2019}. 
So far, there are many definitions for the betweenness relation. The reason for this is that the definition of the transitivity of the ternary relation varies greatly. 
For example, there are definitions based on four points \cite{Huntington-1917, Huntington-1924} and those based on five points \cite{Pitcher-1942, Zedam-2018}.
In 2019, Pérez-Fernández and De Baets \cite{Pérez-2019} gave a definition of betweenness relation which satisfies symmetry in the end points, closure and end-point transitivity. 
After that, Zhang et al. \cite{Zhang2019} represented a betweenness relation in sense of  \cite{Pérez-2019} as a family of order relations and analyze the corresponding family of induced Alexandrov topologies. 
These definitions and conclusions of the betweenness relations are all based on a specific type of four-point transitivity.

The metric space is a type of important and closest-to-Euclidean-space abstract space
and it has extensive applications both in theoretical research and in real life. 
In a metric space, all triples that satisfy the metric triangle identity condition can generate a betweenness relation.
In 1942,  Menger \cite{Menger-1942} introduced the concept of a statistical metric space.
Then Schweizer and Sklar \cite{Schweizer-1960, Schweizer-1983}
developed Menger's statistical metric and introduced a similar concept of a probabilistic metric space,
which refers to the distance between two points  corresponding to a probability distribution function defined on the set of non-negative real numbers. 
Besides, Menger also presented the concept of betweenness in \cite{Menger-1942} and pointed out its possible application fields. 
Shortly afterwards, 
Wald \cite{Wald-1943} extended Menger's triangular inequality and established a theory of betweenness, which showed that 
the properties of Wald's betweenness theory is consistent with the metric betweenness axioms.
Since then, this topic has received very little attention.

Influenced by the development of the fuzzy set theory, various definitions of fuzzy metrics were introduced by 
Kramosil and Michálek,
Grabiec, 
Morsi, 
George and Veeramani,
Mardones-P\'{e}rez and Prada Vicente.
In 1975, Kramosil and Michalek \cite{KramosilMichalek-1975} proved that in a certain sense, the KM-fuzzy metric is equivalent to the probability metric. 
To make the definition more complete, 
Grabiec \cite{Grabiec-1988} added an another condition to the original definition of the KM-fuzzy metric and explained the KM-fuzzy metric is not only a generalization but also a semantic generalization of the classical metric. Subsequently, based on the monotonically decreasing fuzzy numbers, the Morsi's fuzzy metric \cite{Morsi-1988} is introduced, which is equivalent to the complement of the KM-fuzzy metric. In 1994, George and Veeramani \cite{GeorgeVeeramani-1997} modified some conditions, with the aim of making the induced topology Hausdorff. 
However, the topology induced by the GV-fuzzy metric is the classical topology and cannot induce a fuzzy topology. 
In 2015, 
Mardones-P\'{e}rez and Prada Vicente \cite{Mardones-2015}
proposed another variant of the KM-fuzzy metric through triangular norm operators. 
This new definition not only extended the original definition of KM-fuzzy metric but also enabled it to induce fuzzy structures. 

Recently, some scholars have once again paid attention to the research on the fuzzy betweenness relations.
In 2020, Zhang, Pérez-Fernández, and De Baets 
\cite{Zhang2020}
introduced two different definitions of fuzzy betweenness relations, which are $*$-betweenness relations and 
$*$-E betweenness relations.
Soon afterwards, they investigated the issue of how to construct a fuzzy betweenness relation from a crsip metric with continuous Archimedean triangular norm in \cite{Zhang2020-2}.
The description of fuzzy transitivity in these two definitions is based on the classic betweenness relation definition given by Pérez-Fernández, and De Baets \cite{Pérez-2019} and 
it is also a fuzzy generalization of a certain specific type of the four-point transitive properties.
In 2021, Zedam and De Baets  \cite{Zedam2021} proposed the six types of transitivity properties of fuzzy ternary relations based on five points from the perspective of compositions.
In 2022,
Shi \cite{Shiyi-2022} explored the concept of gated sets within the framework of crisp betweenness relations induced by GV-fuzzy metric spaces.
Also, Shi et al. \cite{Shiyi-2023} discussed the relationships between fuzzy betweenness relations and fuzzy interval operators, fuzzy partial orders and fuzzy Peano-Pasch spaces.
In 2025, Jin and Yan provided a method to construct a $*$-betweenness relation and 
$*$-E betweenness relation by the GV-fuzzy metric and strong GV-fuzzy metric, respectively. 
However, through the above introduction, we can see that there are still many gaps in the current research on fuzzy betweenness relations  by fuzzy metrics.  
By introducing the development status of the fuzziness metric, we can see that the KM-fuzzy metric is more in line with the semantic extension of fuzziness compared to the GV-fuzzy metric. 
This is also the starting point of our article. 
we will explore the 
the constructions and the characteristics of both
betweenness relations and fuzzy betweenness relations in KM-fuzzy metric spaces. 

The main purpose of this paper is to introduce two different construction methods for fuzzy betweenness relations induced by a KM-fuzzy metric. 
One approach is achieved by directly using the implication operator. The other method is realized through the corresponding nest of
metrics. 
Furthermore, we demonstrate that 
these two types of fuzzy betweenness relations satisfy the eight kinds of four-point transitivity properties and the six kinds of five-point transitivity properties.

The organization of this paper is as follows. 
In Section 2, we review some meaningful definitions of ternary relations, betweenness relations and KM-fuzzy metrics. 
Among them, the most important thing is that we sort out and summarized the four-point and five-point transitivity properties of a ternary relation. 
In Section 3, we discuss the construction and the characteristics of a family of betweenness relations in KM-fuzzy metric spaces and obtain a nest of induced betweenness relations.
In Section 4, we give two method to construct fuzzy betweenness relations in KM-fuzzy metric spaces.
One of them is directly obtained by using the implication operator. The other is through the corresponding nest of metrics of KM-fuzzy metrics.
Furthermore, we also show that the two types of fuzzy betweenness relations are the same, and they also satisfy the eight kinds of four-point transitivity properties and the six kinds of five-point transitivity properties. 
Finally, the article presents many summary diagrams and an outlook for future research.

\section{\bf Preliminaries}

In this section, we recall some concepts of ternary relations, betweenness relations and KM-fuzzy metrics.
Let $X$ be a non-empty set. 
A ternary relation $T$ is a subset of $X^{3}$. 
Ternary relations have extensive applications in mathematics and logic, computer science and artificial intelligence.

\subsection{Ternary relations}

In this part, we first recall the definitions of reflexivity, symmetry, 
antisymmetry of a ternary relation.
Then we recall the eight kinds of four-point transitivity properties
and the six kinds of five-point transitivity properties 
from the perspective of positional relations and from the perspective of the composition of ternary relations, respectively.

Let $T\subseteq X^{3}$ be a ternary relation. Then $T$ is called
\vskip1mm
\noindent 
(1) reflexive, if $(x, y, y)\in T$ for all $x, y\in X$.

\noindent 
(2) symmetric, if $(x, y, z)\in T$ if and only if $(z, y, x)\in T$  for all $x, y, z\in X$.

\noindent 
(3) antisymmetric, if $(x, y, z)\in T$ and $(x, z, y)\in T$ 
implies $y=z$  for all $x, y, z\in X$.

\subsubsection{Four-point transitivity properties of a ternary relation}

In \cite{Huntington-1917, Huntington-1924}, 
Huntington and Kline gave the eight kinds of four-point transitivity properties of a ternary relation. 
Let $x ,y, s, t \in X$. 
Then the four-point transitivity properties are shown as follows.

\noindent 
(P1) $ (x, s, t)\in T, (s, t, y)\in T $  imply that $(x, s, y)\in T$.

\noindent 
(P2) $ (x, s, t)\in T, (s, y, t)\in T $  imply that $(x, s, y)\in T$.

\noindent 
(P3) $ (x, s, t)\in T, (s, y, t)\in T $  imply that $(x, y, t)\in T$.

\noindent 
(P4) $ (s, x, t)\in T, (s, y, t)\in T $  imply that $(s, x, y)\in T$ or  $(s, y, x)\in T$.

\noindent 
(P5) $ (s, x, t)\in T, (s, y, t)\in T $   imply that $(s, x, y)\in T$ or  $(y, x, t)\in T$.

\noindent 
(P6) $ (x, s, t)\in T, (y, s, t)\in T $  imply that $(x, y, t)\in T$ or  $(y, x, t)\in T$.

\noindent 
(P7)  $ (x, s, t)\in T, (y, s, t)\in T $  imply that $(x, y, s)\in T$ or  $(y, x, s)\in T$.

\noindent 
(P8)  $ (x, s, t)\in T, (y, s, t)\in T $  imply that $(x, y, s)\in T$ or  $(y, x, t)\in T$.

%
%
%
%
%
%
%
%

Why are there exactly these rules for the transitivity-like properties of a ternary relation? To visually demonstrate these eight transitive hypotheses, we place four points $x ,y, s, t$ 
on a straight line, and then use a schematic diagrams (see Figure \ref{summary of four points in a line}) to explain the interrelationships among them.
If we do not specify the direction and position of these four points on the straight line, then according to statistical rules, there can be 24 possible arrangements.
However, if we agree to always present the positional relation of these four points in the direction from $s$ to $t$ (namely, the points $s$ and $t$ being fixed, and points $x$ and $y$ being moving), then these general rules can be reduced to 12 possible arrangements.

Next, we provide a detailed explanation of why the eight types of transitivity-like properties are obtained.
And we can conduct the transitivity analysis according to the following four types.

\noindent
{\bf case 1:} $s$ and $t$ are located in the middle of  $x$ and $y$. From Figure \ref{summary of four points in a line}, we get obtain (P1). 

\noindent
{\bf case 2:} one point of $s$ and $t$ are located in the middle of $x$ and $y$, and the other point is on the side of $x$ and $y$. 
From Figure \ref{summary of four points in a line}, we get (P2) and (P3), which are often used to define the transitivity in betweenness relations \cite{Jin2025, Pérez-2019, Shiyi-2022, Vel-1993, Zhang2019, Zhang2020, Zhang2020-2}.

\noindent
{\bf case 3:} $s$ and $t$ are located on both sides of  $x$ and $y$. From Figure \ref{summary of four points in a line}, we get (P4) and (P5).  

\noindent
{\bf case 4:} $s$ and $t$ are located on the same side of $x$ and $y$. From Figure \ref{summary of four points in a line}, we get (P6),  (P7) and (P8). 

In case 1, the conclusion (1) $\Rightarrow$ (1-1)  is in a similar manner to the conclusion
(1) $\Rightarrow$ (1-2). 
The reason is that the role of $s$ between $x$ and $y$ and the role of $t$ between $x$ and $y$ can essentially be regarded as the same.
Meanwhile, since the moving points $x$ and $y$ are at the endpoints and due to symmetry, we know that the essence of (1$^{*}$) $\Rightarrow$  (1$^{*}$-1) is the same as that of (1) $\Rightarrow$  (1-1).

In case 2, we can further divide it into four categories for analysis. Unlike case 1, the fixed points $s$ and $t$ do not simultaneously lie between the moving points $x$ and $y$. Therefore, we obtain two conclusions. Furthermore, the transitivity (P2) and the transitivity (P3) are somewhat similar to the transitivity of a binary relation.
This makes these two conclusions more widespread.

In case 3 and case 4, 
why are the transitivity conclusions different and why does the case 4 have an additional conclusion? 
Because the fixed points $s$ and $t$ are the endpoints in case 3 and 
considering the symmetry of the endpoints, 
the transitivity of 
`` (3) $\Rightarrow$ (3-2) or (3-2$'$)"  has the same transitive essence as well as (P4),
Therefore 
we only get two transitivity conclusions  
(P4) and (P5).
However, in case 4, the fixed points
$s$ and $t$ are not the endpoints. 
Therefore, we need to consider all three scenarios and obtain three types of transitivity of
(P6), (P7) and (P8).

\begin{figure}[htbp]
\centering  
\includegraphics[height=15.5cm, width=17cm]
{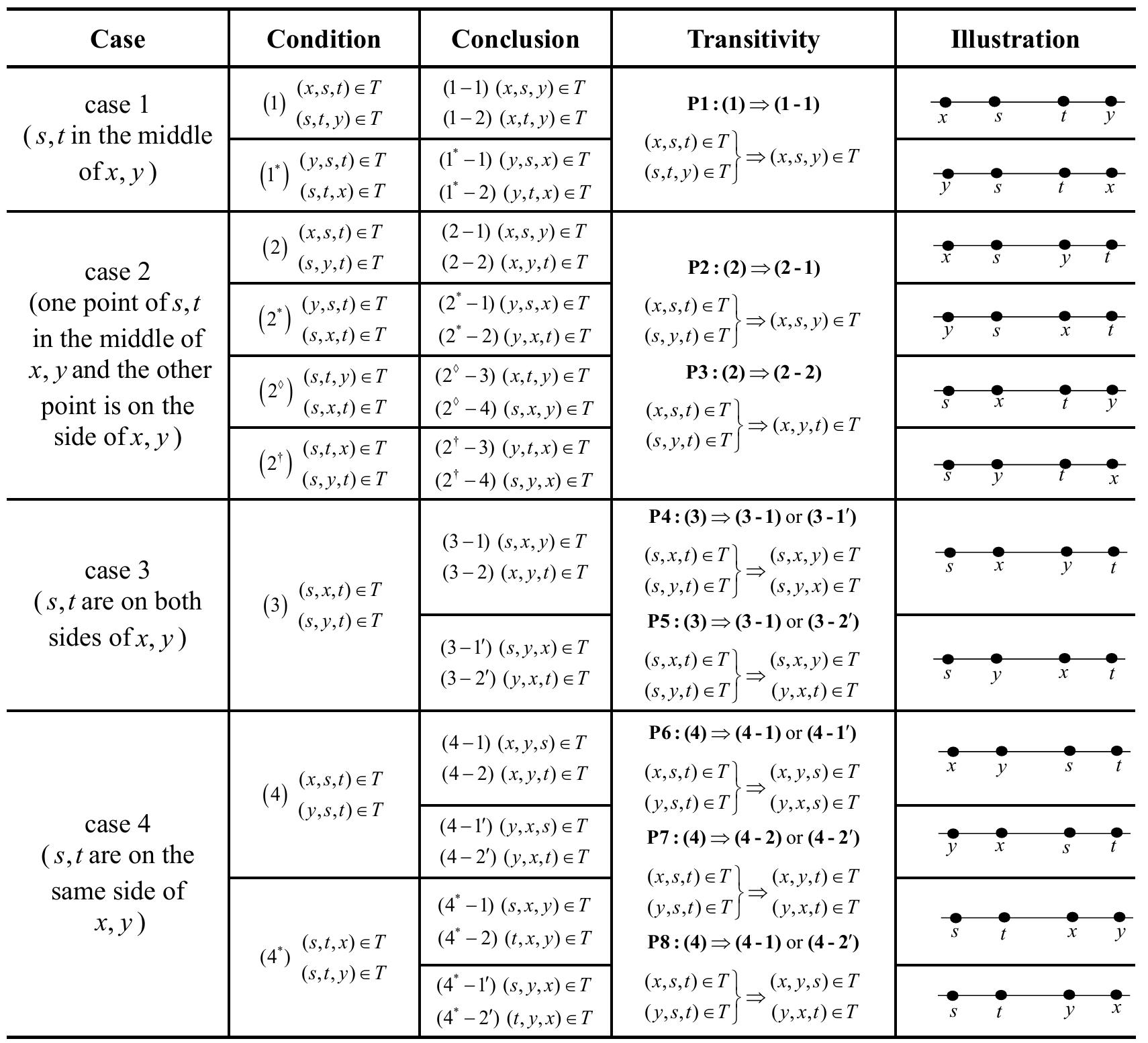}
\caption{A schematic 
diagram illustrating the four-point transitivity properties of ternary relations.} 
\label{summary of four points in a line}
\end{figure}

\subsubsection{Five-point transitivity properties of a ternary relation}

In \cite{Zedam-2020, Zedam-2021-2},
L. Zedam, N. Bakri and B. De Bates gave five-point transitivity properties from the perspective of compositions of ternary relations. 
Firstly, let us recall the definition of two types of 
compositions of a ternary relation with a binary relation introduced  in \cite{Bakri-2021, Zedam-2018}.

\begin{definition}[\cite{Bakri-2021, Zedam-2018}]
\label{compositions of a ternary relation with a binary relation}
Let $T$ be a ternary relation and let $R$ be a binary relation on $X$. Then
the \textit{compositions of $T$ with a binary relation  $R$} is defined by

\noindent 
(1)  
$T\ltimes R 
=\{ (x, y, z)\in X^{3} \mid 
\exists t\in X, ~ (x, y, t)\in T \wedge (t, z)\in R\};$

\noindent 
(2)  
$R\rtimes T
=\{ (x, y, z)\in X^{3} \mid 
\exists t\in X, ~ (x, t)\in R \wedge (t, y, z)\in T\}.$

\end{definition}

Based on the two types of compositions of a ternary relation with a binary relation, we recall the six types of composition of ternary relations by their binary projections, which are introduced by Zedam, Barkat and De Baets in \cite{Zedam-2020,Zedam-2021-2}.

\begin{definition}[\cite{Zedam-2020, Zedam-2021-2}]
\label{compositions of a ternary relation with a binary relation}
Let $T$ and $S$ be ternary relations on $X$. 
Then 
\textit{$\circ_{i}$-compositions ($i\in \{1, ..., 6\}$) of $T$ and $S$} 
are defined as follows.\vskip1mm

\noindent 
(1)  $T \circ_{1}S =T\ltimes P_{l}(S)
=\{ (x, y, z)\in X^{3} \mid 
\exists t, s\in X, ~ (x, y, t)\in T \wedge (s, t, z)\in S\}$;

\noindent 
(2)  $T \circ_{2}S =T\ltimes P_{m}(S)
=\{ (x, y, z)\in X^{3} \mid 
\exists t, s\in X, ~(x, y, t)\in T \wedge (t, s, z)\in S\}$;

\noindent 
\noindent 
(3)  $T \circ_{3}S  =T\ltimes P_{r}(S)
=\{ (x, y, z)\in X^{3} \mid 
\exists t, s\in X, ~(x, y, t)\in T \wedge (t, z, s)\in S\}$;

\noindent 
(4)  $T \circ_{4}S = P_{l}(T)\rtimes S
=\{ (x, y, z)\in X^{3} \mid 
\exists t, s\in X, ~(s, x, t)\in T \wedge (t, y, z)\in S\}$;

\noindent 
(5)  $T \circ_{5}S = P_{m}(T)\rtimes S
=\{ (x, y, z)\in X^{3} \mid 
\exists t, s\in X, ~ (x, s, t)\in T \wedge (t, y, z)\in S\}$;

\noindent 
(6)  $T \circ_{6}S = P_{r}(T)\rtimes S
=\{ (x, y, z)\in X^{3} \mid 
\exists t, s\in X, ~(x, s, t)\in T \wedge (s, y, z)\in S\}$;

\noindent
Among them,  
$P_{l}(T)=\{(x, y)\in X^{2} \mid \exists z\in X,~ (z, x, y)\in T\}$ represents the left projection of $T$.
$P_{m}(T)=\{(x, y)\in X^{2} \mid \exists z\in X,~ (x, z, y)\in T\}$ represents the middle projection of $T$.
And
$P_{r}(T)=\{(x, y)\in X^{2} \mid \exists z\in X,~ (x, y, z)\in T\}$ represents the right projection of $T$.
\end{definition}

The transitivity of a binary relation $R$ is defined as $R \circ R \subseteq R$,  $ \circ $ is 
the usual composition of binary relations. 
Similarly, we can present the five-point transitivity properties of a ternary relation from the perspective of the composition of ternary relations. 

\begin{definition}[\cite{ Zedam-2020,Zedam-2021-2}]
\label{five-point transitivity of a ternary relation}
Let $T$ be a ternary relations on $X$. 
Then $T$ is called 
\textit{$\circ_{i}$-transitive} 
if $T \circ_{i}T \subseteq T$ for any $i\in \{1, ..., 6\}$.
To be precise, let $x ,y, s, t \in X$. 
Then the five-point transitivity properties are shown as follows.

\noindent 
(T1) $ (x, y, t)\in T, (s, t, z)\in T $  imply that $(x, y, z)\in T$.

\noindent 
(T2) $ (x, y, t)\in T, (t, s, z)\in T $  imply that $(x, y, z)\in T$.

\noindent 
(T3) $ (x, y, t)\in T, (t, z, s)\in T $  imply that $(x, y, z)\in T$.

\noindent 
(T4) $ (s, x, t)\in T, (t, y, z)\in T $  imply that $(x, y, z)\in T$.

\noindent 
(T5) $ (x, s, t)\in T, (t, y, z)\in T $   imply that $(x, y, z)\in T$.

\noindent 
(T6) $ (x, s, t)\in T, (s, y, z)\in T $  imply that $(x, y, z)\in T$.
\end{definition}

In order to contrast with the four-point transitivity and also to visually show the five-point transitivity, we create a summary diagram (see Figure \ref{summary of five points in a line}).
It is particularly important to note that: 

\noindent
i) when $s = y$, (T1) of the five-point transitivity is exactly (P1) of the fourth-point transitivity, that is to say, (T1) is the generalization of (P1). 

\noindent
ii) when $t = z$, (T6) of the four-point transitivity is exactly (P3) of the fourth-point transitivity, that is to say, (T6) is the generalization of (P3).

\noindent
iii) No matter whether $s$ or $t$ takes the value of $x$ or $y$, (T2)-(T5) cannot be reduced to the four-point transitivity. That is to say, (T2)-(T5) are distinct from the four-point transitivity and possess their own unique five-point transitivity.

\begin{figure}[htbp]
\centering  
\includegraphics[height=7cm, width=17cm]
{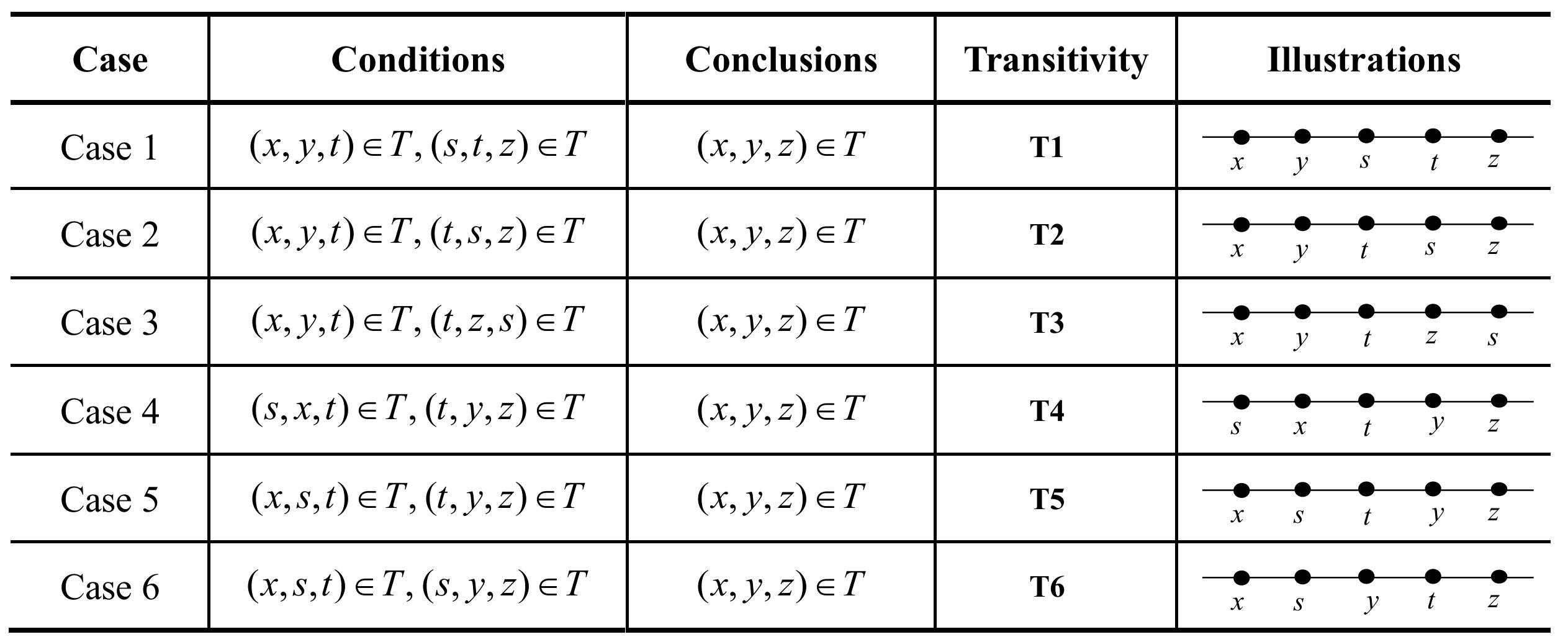}
\caption{A schematic 
diagram illustrating the five-point transitivity properties of ternary relations.} 
\label{summary of five points in a line}
\end{figure}

\subsection{Betweenness relations}

In this part, we focus on discussing a very
important type of the ternary relation,
that is, 
the betweenness relation. 
And we also introduce its connections with various mathematical structures.

The concept of a betweenness relation was first introduced by Pasch \cite{Pasch-1882},  further developed by Huntington and Kline \cite{Huntington-1917,Huntington-1924}.
Its introduction mainly aims to provide an axiomatic description of the geometric concept that 
``one element is located between two other elements".
It should be noted that currently there are many different types of definitions for the betweenness relation \cite{HedlÍková-1983, Vel-1993, Pérez-2019}, mainly because the transitivity of its ternary relationship varies greatly. Here, we adopt the definition  introduced by Pérez-Fernández and B. De Baets in \cite{Pérez-2019}.

\begin{definition}[\cite{Pérez-2019}]
\label{definition of betweenness relation}
A ternary relation $B\subseteq X^{3}$ is called a
\textit{betweenness relation}, 
if it satisfies the following conditions:
\vskip1mm
\noindent 
(B1) Symmetry in the end points: 
$\forall x, y, z\in X, ~(x, y, z)\in B \Leftrightarrow 
(z, y, x)\in B$.

\noindent 
(B2) Closure:
$\forall x, y, z\in X, ~(x, y, z)\in B \wedge
(x, z, y)\in B \Leftrightarrow y=z$.

\noindent 
(B3) End-point transitivity: 
$\forall o, x, y, z\in X, ~(o, x, y)\in B \wedge
(o, y, z)\in B \Rightarrow (o, x, z)\in B $.
\end{definition}

\begin{remark} \label{remark of betweenness relations}
\noindent (1)
The closure property  is equivalent to the followings:
\vskip1mm

\noindent 
(B4) Reflexivity:
$\forall x, y\in X, ~(x, y, y)\in B$.

\noindent 
(B5) Antisymmetry:
$\forall x, y, z\in X, ~(x, y, z)\in B \wedge
(x, z, y)\in B \Rightarrow y=z$.

\noindent (2)
Under the property of symmetry,
the end-point transitivity is actually the (P3) of the four-point transitivity. 
Here we provide a brief explanation as to why (P3) was chosen as the transitivity of the betweenness relation. 
In \cite{Huntington-1917}, Huntington and Kline 
presented four basic postulates for a betweenness relation, that is, 

\noindent 
A) Symmetry: $(a, x, b) \in B$ if and only if $(b, x, a) \in B$ for all $a, b, x\in X$.

\noindent 
B) For all $a, b, c\in X$, at least one of the following six conditions $(a, b, c) \in B, (a, c, b) \in B, (b, a, c) \in B, (b, c, a) \in B, (c, a, b) \in B, (c, b, a) \in B$ must be true.

\noindent 
C) Antisymmetry: $(a, x, y) \in B$ and $(a, y, x) \in B$ imply that $x=y$, for all $a, x, y\in X$.

\noindent 
D) Reflexivity: $(x, y, y) \in B$  for all $x, y\in X$.

Under these four basic postulates,
they had proved that (P6), (P7), and (P8) can be derived from (P2). And (P2) combined with (P1) can lead to (P4) and (P5). 
Also, (P2) combined with (P1) can lead to (P3).
All of them had been proved and shown in Table 1 of  \cite{Huntington-1917}.
Therefore, (P3) is the most important and crucial one in the four-point transitivity properties.
\end{remark}

\begin{example}\label{example of betweenness relation of order}
Let $(X, \leq)$ be a partially ordered set.
Then a \textit{order-betweenness relation} $B_{\leq}\subseteq X^{3}$ is defined by
$$B_{\leq}=\{
(x, y, z)\in X^{3} \mid (x=y)\vee (y=z)\vee (x\leq y\leq z)\vee (z\leq y\leq x)\}.$$
\end{example}

\begin{example}\label{example of betweenness relation of lattice}
Let $(X, \vee, \wedge)$ be a lattice structure.
Then a \textit{lattice-betweenness relation} $B_{\vee, \wedge}\subseteq X^{3}$ is defined by
$$B_{\vee, \wedge}=\{
(x, y, z)\in X^{3} \mid  (x\wedge y)\vee (y\wedge z)
= y = (x\vee y)\wedge (y\vee z)\}.$$
\end{example}

\begin{example}\label{example of betweenness relation of metric}
Let $(X, d)$ be a metric space.
Then a \textit{metric-betweenness relation} $B_{d}\subseteq X^{3}$ is defined by
$$B_{d}=\{
(x, y, z)\in X^{3} \mid d(x, z)=d(x, y)+d(y, z)\}.$$
Since $d(x, z) \leq d(x, y)+d(y, z)$ is trivial, 
it follows that $$B_{d}=\{
(x, y, z)\in X^{3} \mid d(x, z)=d(x, y)+d(y, z)\}
=\{
(x, y, z)\in X^{3} \mid d(x, z)\geq d(x, y)+d(y, z)\}.$$
It is particularly important that we can prove the metric-betweenness relation $B_{d}$ also satisfies the transitivity properties of (P1)-(P8) and (T1)-(T6).
\end{example}

There are many other examples existing in mathematical structures, such as median algebra, vector space, hypergroup, etc. 
However, in this article, subsequently we will focus on the betweenness relations induced by metrics and the fuzzy betweenness relations induced by fuzzy metrics.

\subsection{KM-fuzzy metric spaces}

In this part, we first recall the definitions of a metric and a $t$-norm. 
And then we introduce a variant of the original definition of 
Kramosil and Michalek's fuzzy metric, which now called KM-fuzzy metric.

\begin{definition}
[\cite{Frechet-1906}]\label{Definition of a metric}
A \textit{metric} on a set $X$ is a mapping 
$d: X\times X \rightarrow [0, \infty)$ satisfying the followings: 
(1) $d(x, y)\geq 0$;
(2) $d(x, y)=0 \Leftrightarrow  x=y$;
(3) $d(x, y)=d(y, x)$;
(4) $d(x, z)\leq d(x, y)+d(y, z)$ 
for any $x, y, z \in X$ .
And the pair $(X, d)$ is called a \textit{metric space}.
\end{definition}

\begin{definition}
[\cite{Klement-2000,Schweizer-1983}]\label{Definition of a t-norm}
A \textit{t-norm} on $[0, 1]$ is a binary operation 
$*: [0, 1]\times [0, 1] \rightarrow [0, 1]$ satisfying the followings: 
(1) $a*b=b*a$;
(2) $a*(b*c)=(a*b)*c$;
(3) $a*1=a$;
(4) $a\leq b \Rightarrow a*c\leq b*c$.
for all $ a, b, c\in [0, 1]$.
Obviously, $a*0=0$, $a*b\leq a\wedge b$ for any $a, b\in [0, 1]$.


The most common examples of  \textit{t-norms} are as follows.
(i) the minimum \textit{t-norm}: $a*_{M}b=a\wedge b$.
(ii) the product  \textit{t-norm}: $a*_{P}b=a \cdot b$.
(iii) the Łukasiewicz  \textit{t-norm}: $a*_{L}b=0 \vee (a+b-1)$.
Note that $*=\wedge$ is the strongest \textit{t-norm}.
\end{definition} 

Next, we recall the definition of a KM-fuzzy metric, 
and we also provide a detailed explanation of the reasons for each condition in the definition of KM fuzzy metric. 
First of all, we can equivalently regard a metric as a mapping
$\chi_{d}: X\times X\times[0, \infty)\rightarrow \{0, 1\}$ defined by
\begin{equation*}
\chi_{d}(x, y, t)=
\left\{
\begin{split}
&1,          && t>d(x, y);\\
&0,          && t\leq d(x, y).
\end{split}
\right.
\end{equation*}

\noindent Then  $\chi_{d}$ satisfies the followings: $\forall x, y, z\in X$, 
$\forall s, r, t\in [0, \infty)$,

\noindent (1) $\chi_{d}(x, y, 0)=0$;

\noindent (2) $\forall t>0, ~\chi_{d}(x, y, t)=1\Leftrightarrow x=y$;

\noindent (3) $\chi_{d}(x, y, t)=\chi_{d}(y, x, t)$;

\noindent (4) $\chi_{d}(x, y, s) * \chi_{d}(y, z, r)\leq \chi_{d}(x, z, s+r)$; \vskip1mm

\noindent (5)  $\bigvee\limits_{s<t}\chi_{d}(x, y, s)=\chi_{d}(x, y, t)$;
\vskip1mm

\noindent (6)  $\bigvee\limits_{t>0}\chi_{d}(x, y, t)=1$.
\vskip1mm

The conditions (5) and (6) are the corresponding generalizations of the following properties that naturally hold in the classical case, that is, 
``$\forall s<t, d(x, y)<s \Rightarrow d(x, y)<t$ ", and 
``$\exists t>0~~ s.t.~d(x, y)<t$ ", respectively.
Actually, (5) and (6) are naturally occurring inherent properties.
However, if we are to extend the definition of a metric to the fuzzy case, these conditions cannot be ignored, even they are very important in some conclusions.

%

Based on the above explanations and analyses, we give the definition of a KM-fuzzy metric in the following, which is a variant of the original definition of Kramosil and Michalek \cite{KramosilMichalek-1975}.

\begin{definition}[\cite{Mardones-2012}]
\label{Definition of KM-fuzzy metric}
A \textit{KM-fuzzy metric} is a mapping $M: X\times X\times [0, \infty)\rightarrow [0, 1]$ 
satisfying the followings:
$\forall x, y, z \in X$, $\forall t, s\in[0, \infty)$,

\noindent (FM1)~~ $M(x, y, 0)=0$;

\noindent (FM2)~~ $\forall t>0$,  
$M(x, y, t)=1 \Leftrightarrow  x=y$.

\noindent (FM3)~~ $M(x, y, t)=M(y, x, t)$;

\noindent (FM4)~~ $M(x, y, s)* M(y, z, r)\leq M(x, z, s+r)$;

\noindent (FM5)~~ $M(x, y, -): [0, \infty)\rightarrow [0,1]$ is left-continuours;

\noindent (FM6)~~ $lim_{t\rightarrow \infty}M(x, y, t)=1$.

\noindent
Then the triple $(X, M, *)$ is called a \textit{KM-fuzzy metric space}.
And  $M(x, y, t)$ is interpreted as the degree to which $d(x, y)<t$.
If $[0, 1]$ is reduced to $\{0, 1\}$, then Definition \ref{Definition of KM-fuzzy metric} reduces to the Definition \ref{Definition of a metric}
of the classical metric.
\end{definition}

\begin{remark}
\noindent (1) It is easy to see that (FM4) is equivalent to the following condition
\vskip1mm
\noindent (FM4$^{\dagger}$)~~ $\bigvee\limits_{s+r=t}M(x, y, s)* M(y, z, r)\leq M(x, z, t)$.
\vskip1mm
\noindent
This is very important. We will utilize this condition to construct a fuzzy betweenness relation 
induced by a KM-fuzzy metric.
\vskip1mm

\noindent (2) From conditions (FM2) and (FM4), it can be concluded that $M(x,y,-): [0, \infty)\rightarrow [0, 1]$ 
is non-decreasing. So (FM5) and (FM6) are equivalent to the following condition
\vskip1mm
\noindent (FM5$^{\dagger}$)~~$\bigvee_{s<t}M(x, y, s)=M(x, y, t)$;
\vskip1mm
\noindent (FM6$^{\dagger}$)~~$\bigvee_{t>0}M(x, y, t)=1$.
\end{remark}

In 1994, George and Veeramani \cite{GeorgeVeeramani-1997} mainly modified the conditions (FM1) and (FM5) in the definition of KM-fuzzy metric. 
The modified conditions are as follows: 

\noindent (FGV1)~~ $\forall t>0,~M(x, y, t)>0$ 

\noindent (FGV5)~~ $M(x, y, -): [0, \infty)\rightarrow [0,1]$ is continuours;

\noindent
The conditions for these modifications in the GV-fuzzy metric are merely to ensure that the induced classical topology is Hausdorff, without any background of fuzzy semantic extension.
And the GV-fuzzy metric cannot induce a fuzzy topology. 

From the above description and analysis, we can gain a deeper understanding of the KM-fuzzy metric. 
Compared with the GV-fuzzy metric, the KM-fuzzy metric is more in line with the extension of the fuzzy logic and semantics of the classical metric.



\begin{example}[\cite{Schweizer-1960}]\label{example of KM-fuzzy metric}
Given a metric space
$(X, d)$ and a non-decreasing, left-continuous function $G: [0, \infty)\rightarrow [0, 1]$ with 
$G(0)=0$ and $lim_{t\rightarrow \infty}G(t)=1$, define a mapping 
$M_{d}: X\times X\times [0, \infty)\rightarrow [0, 1]$ by $\forall x, y\in X, ~t\in [0, \infty)$,
\begin{equation*}
M_{d}(x, y, t)=
\left\{
\begin{split}
&G\left(\frac{t}{d(x, y)}\right),        &&~~ if x\neq y;\\
&0,          &&~~ if~ x=y,~t=0;\\
&1,          && ~~if ~x=y,~t>0.
\end{split}
\right.
\end{equation*}
Then $(X, M_{d}, *)$ is a KM-fuzzy metric space under any $t$-norm.  
If $G(t) = \frac{t}{t+1}$, then we obtain a very commonly used KM-fuzzy metric that is generally referred to as the standard one, that is, 
\begin{equation*}
M_{d}(x, y, t)=
\left\{
\begin{split}
&\frac{t}{t+d(x, y)},          && if ~t>0;\\
&0,          && if ~t=0.
\end{split}
\right.
\end{equation*}
\end{example}


\begin{example}[\cite{Zhongyu-2020}]
\label{example2 of KM-fuzzy metric}
Given a normed space
$(X, \|~ \|)$.
Define a mapping 
$M_{\|~ \|}: X\times X\times [0, \infty)\rightarrow [0, 1]$ by $\forall x, y\in X, ~t\in [0, \infty)$,
\begin{equation*}
M_{\|~ \|}(x, y, t)=
\left\{
\begin{split}
&e^{-\frac{\|x-y\|}{t}},          && if ~t>0;\\
&0,          && if ~t=0.
\end{split}
\right.
\end{equation*}
Then $(X, M_{\|~ \|}, \cdot)$ is a KM-fuzzy metric space under
the product $t$-norm.  
\end{example}

\section{\bf Betweenness relations in KM-fuzzy metric spaces}

In this section, we will explain how to derive a family of betweenness relations $\mathcal{B}^{M}=\{B_{d_{a}^{M}} : a\in (0,1)\}$ from a KM-fuzzy metric $M$ and 
$B_{d_{a}^{M}}=\bigcup_{a\leq b}B_{d_{b}^{M}}$. 
Before that, we need to discuss the relationship 
between a KM-fuzzy metric and its corresponding classical metrics.

\subsection{One-to-one correspondence between a KM-fuzzy metric and its nests of metrics}

Throughout this subsection,
unless otherwise stated, the definition of KM-fuzzy metrics is adopted to take the minimum $t$-norm,
that is, $*=\wedge$.

\begin{definition}[\cite{Mardones-2012}]\label{Definition of LSC}
A family of metrics $\{d_{a}: a\in (0, 1)\}$ is called a \textit{nest of metrics} (or called lower semicontinuous) 
if $d_{a}=\bigwedge_{b>a}d_{b}$ for any $a\in (0, 1)$.
\end{definition}

\begin{remark}\label{remark of nest of metrics}
It is easy to get 
$d_{a}$ is increasing with respect to $a$, that is, $d_{a}\leq d_{b}$ for any $a\leq b$.
\end{remark}

Although the following conclusions have already been presented in \cite{Mardones-2012}
(with the background of KM-fuzzy pseudometric), we will show that these conclusions are also valid in the KM-fuzzy metric space through different proof methods.

In the following theorem, we will present that 
a family of classical metrics can be generated by a KM-fuzzy metric and the set of these classical metrics form a nest of metrics.

\begin{theorem}\label{theorem of a nest of metric induced by a KM-fuzzy metric}
Let $(X, M, \wedge)$ be a KM-fuzzy metric space. For any $a\in (0, 1)$,
define a mapping $d_{a}^{M}: X\times X\rightarrow [0, \infty)$ by $\forall x, y\in X$,
$$d_{a}^{M}(x, y)=\bigvee\{t\in [0, \infty) \mid M(x, y, t)\leq a\}.$$
Then
\vskip1mm
\noindent {\rm (1)}
$M(x, y, t)\leq a \Leftrightarrow d_{a}^{M}(x, y)\geq t$, ~i.e.,~
$M (x, y, t)> a \Leftrightarrow d_{a}^{M}(x, y)<t$.
\vskip1mm
\noindent {\rm (2)}
$d_{a}^{M}(x, y)=\bigwedge\{t\in [0, \infty) \mid M(x, y, t)> a\}$.
\vskip1mm
\noindent {\rm (3)} For any $a\in (0, 1)$,
$d_{a}^{M}$ is a metric and $d_{a}^{M}=\bigwedge_{b>a}d_{b}^{M}$.
\vskip1mm
\end{theorem}

\begin{proof}
\noindent (1)
By the definition of $d_{a}^{M}$, it is easily to seen that $M(x, y, t)\leq a$ implies
$d_{a}^{M}(x, y)\geq t$.
On the other hand, suppose that $d_{a}^{M}(x, y)\geq t$. Take any 
$s<t\leq d_{a}^{M}(x, y)=\bigvee\{t\in [0, \infty) \mid M(x, y, t)\leq a\}$, there exists some $t_{0}$ with
$M(x, y, t_{0})\leq a$ such that $s< t_{0}$.
Then $M(x, y, s)\leq M(x, y, t_{0})\leq a$.
So $\bigvee_{s<t}M(x, y, s)=M(x, y, t)\leq a$.

\noindent (2)
It can be easily concluded from conclusion (1).

\noindent (3)
We firstly prove $d_{a}^{M}$ is a metric for any
$a\in (0, 1)$ and then prove $d_{a}^{M}=\bigwedge_{b>a}d_{b}^{M}$.

\noindent
i) $d_{a}^{M}(x, x)=\bigvee\emptyset=0$.

\noindent
ii) If $d_{a}^{M}(x, y)=0$, then 
$d_{a}^{M}(x, y)<t$ for any $t>0$.
By (1), we know $M(x, y,t)>a$ for any 
$a\in (0, 1)$.
Then $M(x, y,t)=1$ for any $t>0$.
By (FM2), we get $x=y$.

\noindent
iii)  $d_{a}^{M}(x, y)=d_{a}^{M}(y, x)$ is trivial. 

\noindent
iv) we need to prove 
$d_{a}^{M}(x, z)\leq d_{a}^{M}(x, y)+d_{a}^{M}(y, z)$.
Take any $t+s>0$ and $d_{a}^{M}(x, y)+d_{a}^{M}(y, z)<t+s$.
Then $d_{a}^{M}(x, y)<t$ and $d_{a}^{M}(y, z)<s$, 
which imply $M(x, y, t)>a$ and $M(y, z, s)>a$.
So $M(x, z, t+s)\geq M(x, y, t)\wedge M(y, z, s)>a$.
Hence $d_{a}^{M}(x, z)<t+s$.
By the arbitrariness of $t$ and $s$, we get 
$d_{a}^{M}(x, z)\leq d_{a}^{M}(x, y)+d_{a}^{M}(y, z)$.
Combing (i)-(iv), we get 
$d_{a}^{M}$ is a metric for any $a\in (0, 1)$.

\noindent
v) 
From the definition of $d_{a}^{M}$, 
it is easy to see $d_{b}^{M}\geq d_{a}^{M}$ for any $a<b$.
So $\bigwedge_{b>a}d_{b}^{M}\geq d_{a}^{M}$. 
On the other hand, take any $t>0$ with $d_{a}^{M}(x, y)<t$. 
Then $M(x, y, t)>a$ and
there exist $b>a$ such that $M(x, y,t)>b>a$.
So $d_{b}^{M}(x, y)<t$.
Hence $\bigwedge_{b>a}d_{b}^{M}(x, y)<t$. 
By the arbitrariness of $t$, we get $\bigwedge_{b>a}d_{b}^{M}\leq d_{a}^{M}$.
Therefore $\bigwedge_{b>a}d_{b}^{M}=d_{a}^{M}$.
\end{proof} 

Next, we shall show that a KM-fuzzy metric can also be constructed by a nest of metrics.

\begin{theorem}\label{theorem of a KM-fuzzy metric induced by a nest of metrics}
Let $\mathcal{D}=\{d_{a}: a\in (0, 1)\}$ be a nest of metrics. 
Define a mapping $M^{\mathcal{D}}: X\times X\rightarrow [0, \infty)\rightarrow [0, 1]$ by $\forall x, y\in X$, $\forall t\in [0, \infty)$,
$$M^{\mathcal{D}}(x, y, t)=\bigwedge\{a\in (0, 1) \mid d_{a}(x, y)\geq t\}.$$
Then
\vskip1mm
\noindent {\rm (1)}
$d_{a}(x, y)\geq t \Leftrightarrow M^{\mathcal{D}}(x, y, t)\leq a $, ~i.e.,~
$d_{a}(x, y)< t \Leftrightarrow M^{\mathcal{D}}(x, y, t)> a $
\vskip1mm
\noindent {\rm (2)}
$M^{\mathcal{D}}(x, y, t)=\bigvee\{a\in (0, 1) \mid d_{a}(x, y)< t\}$.
\vskip1mm
\noindent {\rm (3)}  
$(X, M^{\mathcal{D}}, \wedge)$ is a KM-fuzzy metric space. 
\end{theorem}

\begin{proof}
\noindent (1)
By the definition of $M^{\mathcal{D}}$, 
it is easily to seen that $d_{a}(x, y)\geq t$ implies
$M^{\mathcal{D}}(x, y, t)\leq a$.
On the other hand, 
suppose that $M^{\mathcal{D}}(x, y, t)\leq a$. 
Take any $b>a\geq M^{\mathcal{D}}(x, y, t)=\bigwedge\{a\in (0, 1) \mid d_{a}(x, y)\geq t\}$, there exists some $a_{0}\in (0, 1)$ with $d_{a_{0}}(x, y)\geq t$ such that $ b>a_{0}$.
So $d_{b}(x, y)\geq d_{a_{0}}(x, y)\geq t$.
Since $d_{a}=\bigwedge_{b>a}d_{b}$,
it follows that $\bigwedge_{b>a}d_{b}(x, y)=d_{a}(x, y)\geq t$.

\noindent (2)
It can be easily obtained from conclusion (1).

\noindent (3) We need to check that $M^{\mathcal{D}}$ satisfy the conditions
(FM1)-(FM6).

\noindent
(FM1) $M^{\mathcal{D}}(x, y, 0)=\bigwedge_{a\in (0,1)}a=0$.

\noindent
(FM2) For any $t>0$,
$M^{\mathcal{D}}(x, x, t)=\bigwedge\emptyset=1$.
On the other hand, 
if $M^{\mathcal{D}}(x, y, t)=1$ for any  $t>0$,
then $M^{\mathcal{D}}(x, y, t)>a$. 
By (1), we get $d_{a}(x, y)<t$ for any $t>0$.
This implies $d_{a}(x, y)=0$.
So $x=y$.

\noindent
(FM3)  $M^{\mathcal{D}}(x, y, t)=M^{\mathcal{D}}(y, x, t)$ is trivially. 

\noindent
(FM4) we need to prove 
$M^{\mathcal{D}}(x, y, t))\wedge M^{\mathcal{D}}(y, z, s))\leq M^{\mathcal{D}}(x, z, t+s))$.
Take any $a>0$ with $M^{\mathcal{D}}(x, z, t+s))\leq a$. Then $d_{a}(x, z)\geq t+s$.
Since 
$d_{a}(x, y)+d_{a}(y, z)\geq d_{a}(x, z)\geq t+s$,
we know $d_{a}(x, y)\geq t$ and 
$d_{a}(y, z)\geq s$.
Then $M^{\mathcal{D}}(x, y, t))\leq a$ and 
$M^{\mathcal{D}}(y, z, s))\leq a$.
So $M^{\mathcal{D}}(x, y, t))\wedge M^{\mathcal{D}}(y, z, s))\leq a$.
By the arbitrariness of $t$, we get 
$M^{\mathcal{D}}(x, y, t))\wedge M^{\mathcal{D}}(y, z, s))\leq M^{\mathcal{D}}(x, z, t+s))$.

\noindent
(FM5) we prove the equivalent conditions (FM5$^{\dagger}$), that is,
$\bigvee_{s<t}M^{\mathcal{D}}(x, y, s)=M^{\mathcal{D}}(x, y, t)$.
On one hand, $\{a\in (0, 1) \mid d_{a}(x, y)\geq t\}\subseteq \{a\in (0, 1) \mid d_{a}(x, y)\geq s\}$
for any $s<t$.
Then $M^{\mathcal{D}}(x, y, t)\geq M^{\mathcal{D}}(x, y, s)$.
So $M^{\mathcal{D}}(x, y, t)\geq \bigvee_{s<t}M^{\mathcal{D}}(x, y, s)$.
On the other hand, take any $a>0$ with $a<M^{\mathcal{D}}(x, y, t)$. 
Then $d_{a}(x, y)<t$ 
and there exists some $s<t$ such that
$d_{a}(x, y)<s<t$.
So $a<M^{\mathcal{D}}(x, y, s)\leq \bigvee_{s<t}M^{\mathcal{D}}(x, y, s)$.
By the arbitrariness of $t$, we get 
$M^{\mathcal{D}}(x, y, t)\leq \bigvee_{s<t}M^{\mathcal{D}}(x, y, s)$.
Therefore $M^{\mathcal{D}}(x, y, t)= \bigvee_{s<t}M^{\mathcal{D}}(x, y, s)$.

\noindent
(FM6) we prove the equivalent conditions (FM6$^{\dagger}$), that is, $\bigvee_{t>0}M^{\mathcal{D}}(x, y, t)=1$.
From the definition of $M^{\mathcal{D}}$, 
we know
$\bigvee_{t>0}M^{\mathcal{D}}(x, y, t)=\bigvee_{t>0}\bigwedge \{
a\in (0,1) \mid d_{a}(x, y)\geq t\}=\bigwedge \emptyset=1$.
\end{proof} 

Finally, we shall show that there is a one-to-one correspondence between a KM-fuzzy metric and its nests of metrics through the following theorem and the summary diagram (see Figure \ref{summary2 of M=MDM}).

\begin{figure}[htbp]
\centering  
\includegraphics[height=5cm, width=17cm]
{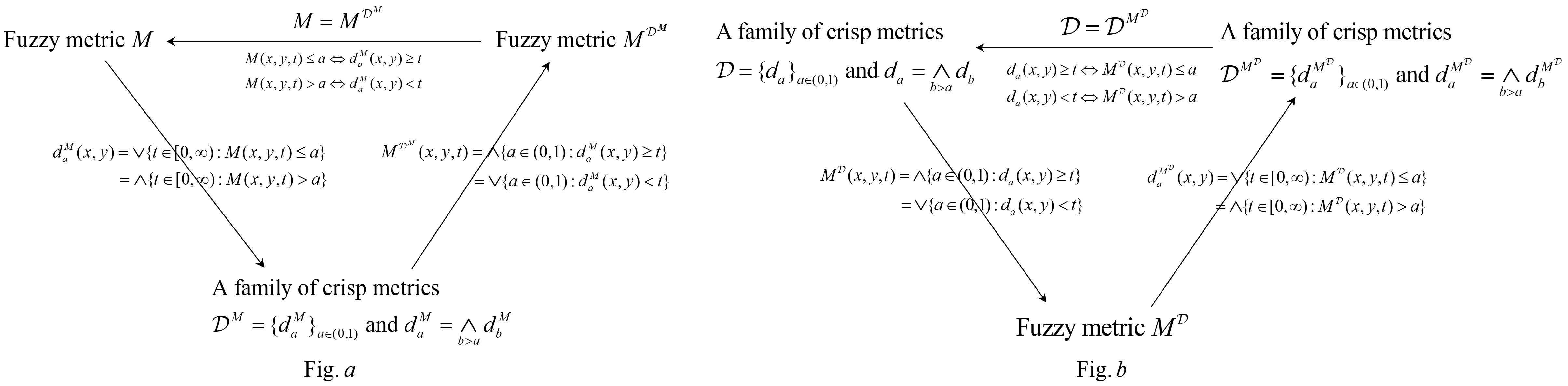}
\caption{One-to-one correspondence between a KM-fuzzy metric and its corresponding nests of metrics.} 
\label{summary2 of M=MDM}
\end{figure}

\begin{theorem}\label{one-to-one correspondence between a KM-fuzzy metric and its nests of metrics}
Let $(X, M, \wedge)$ be a KM-fuzzy metric space and
let $\mathcal{D}=\{d_{a}: a\in (0, 1)\}$ be a nest of metrics. 
Then 
$M^{\mathcal{D}^{M}}=M$ and 
$\mathcal{D}^{M^{\mathcal{D}}}=\mathcal{D}$.
\end{theorem}

\begin{proof}
Based on Theorem \ref{theorem of a nest of metric induced by a KM-fuzzy metric} and Theorem 
\ref{theorem of a KM-fuzzy metric induced by a nest of metrics}, we can easily get the following conclusions. 
For all $x, y, z\in X$,
for any $a\in (0,1)$ and for any $t\in [0, \infty)$,
\begin{eqnarray*}
&&M^{\mathcal{D}^{M}}(x, y, z)\leq a
\Leftrightarrow d_{a}^{M}(x, y)\geq t
\Leftrightarrow M(x, y, z)\leq a\\
&&d_{a}^{M^{\mathcal{D}}}(x, y)\geq t
\Leftrightarrow M^{\mathcal{D}}(x, y, t)\leq a
\Leftrightarrow d_{a}(x, y)\geq t.
\end{eqnarray*}
Therefore $M^{\mathcal{D}^{M}}=M$ and 
$\mathcal{D}^{M^{\mathcal{D}}}=\mathcal{D}$.
\end{proof}


\subsection{Betweenness relations in KM-fuzzy metric spaces}

In the first part, we have showed that 
a family of metrics $\{d_{a}^{M}\}_{a\in (0, 1)}$ can be generated by a KM-fuzzy metric.
Based on that, we will obtain a family of betweenness relations $\{B_{d_{a}^{M}}\}_{a\in (0, 1)}$  
and these betweenness relations also have the following conclusions.

\begin{theorem}\label{theorem1 of betweenness relations in KM-fuzzy metric spaces}
Let $(X, M, \wedge)$ be a KM-fuzzy metric space. 
For any $a\in (0, 1)$, define 
\begin{eqnarray*}
B_{d_{a}^{M}}
&=&\{(x, y, z)\in X^{3} \mid 
d_{a}^{M}(x, z)=d_{a}^{M}(x, y)+d_{a}^{M}(y, z)\}\\
&=&\{(x, y, z)\in X^{3} \mid 
d_{a}^{M}(x, z)\geq d_{a}^{M}(x, y)+d_{a}^{M}(y, z)\}.
\end{eqnarray*}
Then 

\noindent{\rm (1)}
$B_{d_{a}^{M}}$ is a betweenness relation, where
$d_{a}^{M}(x, y)=\bigvee\{t\in [0, \infty) \mid M(x, y, t)\leq a\}$.

\noindent{\rm (2)} For any $a\leq b$,
$B_{d_{b}^{M}}\subseteq B_{d_{a}^{M}}$.

\noindent{\rm (3)} For any $a\leq b$,
$d_{b}^{M}(x, z)\geq d_{b}^{M}(x, y)+d_{b}^{M}(y, z)
\Rightarrow d_{a}^{M}(x, z)\geq d_{a}^{M}(x, y)+d_{a}^{M}(y, z).$
\end{theorem}

\begin{proof}
\noindent (1)
By Theorem \ref{theorem of a nest of metric induced by a KM-fuzzy metric}, we know 
$d_{a}^{M}$ is a metric. 
Then

\noindent (B1) 
$(x, y, z)\in B_{d_{a}^{M}}\Leftrightarrow (z, y, x)\in B_{d_{a}^{M}}$ is obvious.

\noindent (B4) 
$(x, y, y)\in B_{d_{a}^{M}}$ is trivial.

\noindent (B5) 
If $(x, y, z)\in B_{d_{a}^{M}}$ and $(x, z, y)\in B_{d_{a}^{M}}$, then 
$d_{a}^{M}(x, z)=d_{a}^{M}(x, y)+d_{a}^{M}(y, z)$ and $d_{a}^{M}(x, y)=d_{a}^{M}(x, z)+d_{a}^{M}(z, y)$.
So $d_{a}^{M}(z, y)+d_{a}^{M}(y, z)=0$.
This shows $y=z$.

\noindent (B3) 
If $(o, x, y)\in B_{d_{a}^{M}}$ and $(o, y, z)\in B_{d_{a}^{M}}$, then $d_{a}^{M}(o, y)=d_{a}^{M}(o, x)+d_{a}^{M}(x, y)$ and 
$d_{a}^{M}(o, z)=d_{a}^{M}(o, y)+d_{a}^{M}(y, z)$. Suppose that 
$(o, x, z)\notin B_{d_{a}^{M}}$.
Then $d_{a}^{M}(o, z)<d_{a}^{M}(o, x)+d_{a}^{M}(x, z)\leq d_{a}^{M}(o, x)+d_{a}^{M}(x, y)++d_{a}^{M}(y, z)$.
This implies $d_{a}^{M}(o, y)+d_{a}^{M}(y, z)<d_{a}^{M}(o, x)+d_{a}^{M}(x, y)++d_{a}^{M}(y, z)$.
So $d_{a}^{M}(o, y)<d_{a}^{M}(o, x)+d_{a}^{M}(x, y)=d_{a}^{M}(o, y)$, 
which is a contradiction. 
Hence $(o, x, z)\in B_{d_{a}^{M}}$. 
Since (B4) and (B5) are equivalent to (B2),
it follows that $B_{d_{a}^{M}}$ satisfy the conditions 
(B1)-(B3).
Therefore  $B_{d_{a}^{M}}$ is a betweenness relation.

\noindent (2)
For any $b>a$ and
for any $(x, y, z)\in B_{d_{b}^{M}}$, 
we get 
$d_{b}^{M}(x, z)\geq d_{b}^{M}(x, y)+d_{b}^{M}(y, z)$.
Then $$\bigwedge_{b>a}d_{b}^{M}(x ,z)\geq \bigwedge_{b>a} \left(d_{b}^{M}(x, y)+d_{b}^{M}(y, z)\right) \geq   \left(\bigwedge_{b>a} d_{b}^{M}(x, y)\right)+
\left(\bigwedge_{b>a} d_{b}^{M}(y, z)\right).$$
From Theorem \ref{theorem of a nest of metric induced by a KM-fuzzy metric}, we know
$d_{a}^{M}=\bigwedge_{b>a}d_{b}^{M}$.
Hence $d_{a}^{M}(x, z)\geq d_{a}^{M}(x, y)+d_{a}^{M}(y, z)$. This shows $(x, y, z)\in B_{d_{a}^{M}}$. 
If $b=a$, then $B_{d_{b}^{M}}=B_{d_{a}^{M}}$.
Therefore $B_{d_{b}^{M}}\subseteq B_{d_{a}^{M}}$ for any $a\leq b$.

\noindent (3) It can be easily derived from conclusion (2).
\end{proof} 

Based on Theorem \ref{theorem1 of betweenness relations in KM-fuzzy metric spaces}, it is easy to get the following theorem.

\begin{theorem}\label{theorem2 of betweenness relations in KM-fuzzy metric spaces }
Let $(X, M, \wedge)$ be a KM-fuzzy metric space. 
Then $B_{d_{a}^{M}}=\bigcup_{a\leq b}B_{d_{b}^{M}}$ and the family of betweenness relations 
$\mathcal{B}^{M}=\{B_{d_{a}^{M}} : a\in (0,1)\}$ is call a nest of betweenness relations.
\end{theorem}

At the end of Section 3,
we obtain the following summary diagram (see Figure \ref{summary of betweenness relations in KM-fuzzy metric spaces}).

\begin{figure}[htbp]
\centering  
\includegraphics[height=2cm, width=16cm]
{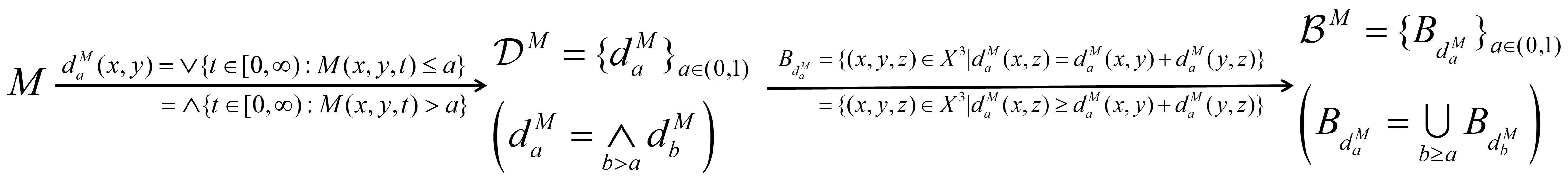}
\caption{Constructions and Characterizations of betweenness relations in KM-fuzzy metric spaces.} 
\label{summary of betweenness relations in KM-fuzzy metric spaces}
\end{figure}

\section{\bf Fuzzy betweenness relations in KM-fuzzy metric spaces}

In this section, we will present two methods for constructing fuzzy betweenness relations in KM-fuzzy metric spaces. One of them is achieved by using the implication operator, where the fuzzy betweenness relation is directly induced by a KM-fuzzy metric. Another approach is to induce an another fuzzy betweenness relation through its corresponding nest of metrics.
Additionally, we will demonstrate that the two induced fuzzy  betweenness relations are equal.

In \cite{Zhang2020-2} 
Zhang et al. introduced the concepts of 
$*$-betweenness relations and strong $*$-betweenness relation.
Subsequently, Shi et al.  \cite{Shiyi-2023} gave a new definition of a fuzzy betweenness relation to establish a bijection between fuzzy geometric interval operators and fuzzy betweenness relations.
All of them can be viewed as the fuzzy generalizations of classical betweenness relations. 

\begin{definition}[\cite{Zhang2020-2}]
\label{the definition of a *-betweenness relation}
Let $*$ be a $t$-norm on $[0,1]$.
A fuzzy ternary relation $B: X^{3}\rightarrow [0, 1]$ is called a \textit{$*$-betweenness relation}, 
if it satisfies the following conditions:
\vskip1mm
\noindent 
(FB1) Symmetry: $B(x, y, z)=B(z, y, x)$
for all $x, y, z\in X$;

\noindent 
(FB2) Reflexivity: $B(x, y, y)=1$
for all $x, y\in X$;

\noindent 
(FB3) Crisp antisymmetry: 
$B(x, y, z)=B(x, z, y)=1$ implies $y=z$
for all $x, y, z\in X$;

\noindent 
(FB4) $*$-transitivity: 
$B(o, x, y)*B(o, y, z)\leq B(o, x, z)$ 
for all $o, x, y, z\in X$;

\noindent
The pari $(X, B)$ is called a \textit{$*$-betweenness set} and the value $B(x, y, z)$ is interpreted as the degree to which $y$ is in between $x$ and $z$.
\noindent
\end{definition}

\begin{definition}[\cite{Zhang2020-2}]
\label{the definition of a *-betweenness relation}
Let $*$ be a $t$-norm on $[0,1]$.
A fuzzy ternary relation $B: X^{3}\rightarrow [0, 1]$ is called a \textit{strong $*$-betweenness relation}, 
if it satisfies the following conditions:
\vskip1mm
\noindent 
(SFB1) Symmetry: $B(x, y, z)=B(z, y, x)$
for all $x, y, z\in X$;

\noindent 
(SFB2) Reflexivity: $B(x, y, y)=1$
for all $x, y\in X$;

\noindent 
(SFB3) $*$-antisymmetry: 
$B(x, y, z)*B(x, z, y)\neq 0$ implies $y=z$
for all $x, y, z\in X$;

\noindent 
(SFB4) $*$-transitivity: 
$B(o, x, y)*B(o, y, z)\leq B(o, x, z)$ 
for all $o, x, y, z\in X$;

\noindent
The pari $(X, B)$ is called a \textit{strong $*$-betweenness set}. 
\end{definition}



\begin{definition}[\cite{Shiyi-2023}]
\label{the definition of a betweenness relation}
A fuzzy ternary relation $B: X^{3}\rightarrow [0, 1]$ is called a \textit{fuzzy betweenness relation}, 
if it satisfies the following conditions:
\vskip1mm
\noindent 
(FBR1) Symmetry: $B(x, y, z)=B(z, y, x)$
for all $x, y, z\in X$;

\noindent 
(FBR2) Reflexivity: $B(x, y, y)=1$
for all $x, y\in X$;

\noindent 
(FBR3) $*$-antisymmetry: $B(x, y, z)*B(x, z, y)\neq 0$ implies $y=z$
for all $x, y, z\in X$;

\noindent 
(FBR4) $*_{1}$-transitivity: $B(o, x, y)*B(o, y, z)\leq B(o, x, z)$ 
for all $o, x, y, z\in X$;

\noindent 
(FBR5)  $*_{2}$-transitivity: $B(o, x, y)*B(o, y, z)\leq B(x, y, z)$ 
for all $o, x, y, z\in X$;

\noindent
The pari $(X, B)$ is called a \textit{fuzzy betweenness set}.
\end{definition}

\begin{remark}
\noindent (1)
(FBR4) is exactly the (FB4) and (SFB4). 
And $*_{1}$-transitivity and $*_{2}$-transitivity are the fuzzy generalizations of (P2) and (P3), respectively.

\noindent (2) 
The differences of these three types of fuzzy betweenness relations lie in the description of antisymmetry. 
(FBR3) is (SFB3), which can be regarded as a generalization of (FB3).
In fact, all these three definitions can, to some extent, be referred to as fuzzy betweenness relations.
\end{remark}

\subsection{A fuzzy betweenness relation $\mathfrak{B}^{M}$ induced by a KM-fuzzy metric $M$}

In this part, 
we shall show that a fuzzy betweenness relation $\mathfrak{B}^{M}$ can be induced by a KM-fuzzy metric $M$ through the implication operator. 
For any $a, b\in [0,1]$, 
define the implication operator $\rightarrow: [0, 1]\times [0,1] \rightarrow [0,1]$ by
$a\rightarrow b =\bigvee \{c \in [0, 1]\mid a*c\leq b\}$.
Then $a*c\leq b  \Leftrightarrow c\leq a \rightarrow b $.

Although the same construction method of $\mathfrak{B}^{M}$ has been presented in \cite{Jin2025}, this paper reaches more conclusions and its proof method is even more simplified.
Also, we will subsequently introduce another method for constructing an another fuzzy betweenness relation in the subsection 4.2 
and discuss the connections between the two construction methods.

\begin{theorem}
\label{a fuzzy betweenness relation1 induced by a KM-fuzzy metric}
Let $(X, M, \wedge)$ be a KM-fuzzy metric space. Define a fuzzy ternary relation $\mathfrak{B}^{M}: X^{3}\rightarrow [0, 1]$ by
$$
\mathfrak{B}^{M}(x, y, z)
=\bigwedge_{t>0}\left(M(x, z, t)\rightarrow
\left(\bigvee_{s+r=t}M(x, y, s)\wedge M(y,z,r)\right)\right).
$$
Then 

\noindent {\rm (1)}
For any $a\in (0,1)$, 
$$a\leq \mathfrak{B}^{M}(x, y, z)
\Leftrightarrow \forall t>0, a\wedge M(x, z, t)\leq \bigvee_{s+r=t}M(x, y, s)\wedge M(y,z,r).$$

\noindent {\rm (2)}
$\mathfrak{B}^{M}$ is a fuzzy betweenness relation. 

\end{theorem}

\begin{proof}
\noindent (1)
It is easily to seen from the property $a*c\leq b  \Leftrightarrow c\leq a \rightarrow b $ of the implication operator. 

\noindent (2) 
We need to prove that $\mathfrak{B}^{M}$ satisfies (FBR1)-(FBR5).
For any $x, y, z\in X$,

\noindent (FBR1)  $\mathfrak{B}^{M}(x, y, z)=\mathfrak{B}^{M}(z, y, x)$ is obvious.

\noindent (FBR2) 
By (FM2) and (FM5$^{\dagger}$),
we get
\begin{eqnarray*}
&&\mathfrak{B}^{M}(x, y, y)
=\bigwedge_{t>0}\left(M(x, y, t)\rightarrow
\left(\bigvee_{s+r=t}M(x, y, s)\wedge M(y, y, r)\right)\right)\\
&=&\bigwedge_{t>0}\left(M(x, y, t)\rightarrow
\left(\bigvee_{s+r=t}M(x, y, s)\right)\right)
=\bigwedge_{t>0}\left(M(x, y, t)\rightarrow
M(x, y, t)\right)=1
\end{eqnarray*}

\noindent (FBR3) 
Suppose that $\mathfrak{B}^{M}(x, y, z)\wedge \mathfrak{B}^{M}(x, z, y)=a\neq 0$.
Then $a\leq \mathfrak{B}^{M}(x, y, z)$ and
$a\leq \mathfrak{B}^{M}(x, z, y)$.
By (1), we get
$a\wedge M(x, z, t)\leq \bigvee_{t_{1}+t_{2}=t}
M(x, y, t_{1})\wedge M(y, z, t_{2})$ for any $t>0$
and $a\wedge M(x, y, r)\leq \bigvee_{r_{1}+r_{2}=r}
M(x, z, r_{1})\wedge M(z, y, r_{2})$ for any $r>0$.
If we assume $y\neq z$, then from (FM2) that
$M(y, z, t)<1$ and $M(z, y, t)<1$ for any $t>0$.
Further, 
\begin{eqnarray*}
&&a*M(x, z, t)
\leq \bigvee_{t_{1}+t_{2}=t}
\left(\left(\bigvee_{r_{1}+r_{2}=t_{1}}
M(x, z, r_{1})\wedge M(z, y, r_{2})\right)
\wedge M(y, z, t_{2})\right)\\
&\leq&\bigvee_{t_{1}+t_{2}=t}
M(x, z, t_{1})\wedge M(y, z, t_{2})
\leq M(x, z, t)\wedge M(y, z, t)
<M(x, z, t).
\end{eqnarray*}
Then $a\wedge M(x, z, t)<M(x, z, t)$ for any $t>0$.
Similarly, we can obtain
$a\wedge M(x, y, r)<M(x, z, r)$ for any $r>0$.
This implies $a<M(x, z, t)$  for any $t>0$
and $a<M(x, y, r)$  for any $r>0$.
From Theorem \ref{theorem of a nest of metric induced by a KM-fuzzy metric}, we get
$d_{a}^{M}(x, z)<t$ for any $t>0$
and $d_{a}^{M}(x, y)<r$  for any $r>0$.
Then $d_{a}^{M}(x, z)=0$
and $d_{a}^{M}(x, y)=0$.
Since $d_{a}^{M}$ is a metric, it follows that
$x=z$ and $x=y$, a contradiction. 
So $y=z$.

\noindent (FBR4) We need to prove 
$\mathfrak{B}^{M}(o, x, y)\wedge 
\mathfrak{B}^{M}(o, y, z)\leq 
\mathfrak{B}^{M}(o, x, z)$.
Take any $a\in (0, 1)$ with $\leq \mathfrak{B}^{M}(o, x, y)\wedge 
\mathfrak{B}^{M}(o, y, z)$.
Then $a\leq \mathfrak{B}^{M}(o, x, y)$ and
$a\leq \mathfrak{B}^{M}(o, y, z)$. 
By (1), we get
$a\wedge M(o, y, t)\leq \bigvee_{t_{1}+t_{2}=t}
M(o, x, t_{1})\wedge M(x, y, t_{2})$ for any $t>0$
and 
$a\wedge M(o, z, r)\leq \bigvee_{r_{1}+r_{2}=r}
M(o, y, r_{1})\wedge M(y, z, r_{2})$ for any $r>0$.
So 
\begin{eqnarray*}
a\wedge M(o, z, r)
&\leq& \bigvee_{r_{1}+r_{2}=r}
\left(\left(\bigvee_{t_{1}+t_{2}=r_{1}}
M(o, x, t_{1})\wedge M(x, y, t_{2})\right)
\wedge M(y, z, r_{2})\right)\\
&=&\bigvee_{r_{1}+r_{2}=r}~\bigvee_{ t_{1}+t_{2}=r_{1}}
M(o, x, t_{1})\wedge M(x, y, t_{2})
\wedge M(y, z, r_{2})\\
&\leq&\bigvee_{t_{1}+t_{2}+r_{2}=r}
M(o, x, t_{1})\wedge M(x, z, t_{2}+r_{2})
=\mathfrak{B}^{M}(o, x, z).
\end{eqnarray*}
From  (1), we get
$a\leq \mathfrak{B}^{M}(o, x, z)$.
By the arbitrariness of $a$, 
$\mathfrak{B}^{M}(o, x, y)\wedge 
\mathfrak{B}^{M}(o, y, z)\leq 
\mathfrak{B}^{M}(o, x, z)$.

\noindent (FBR5) 
The proof is similar to that of (FBR4) and omitted here.

\end{proof}

\begin{remark}
(1) The constrcution of the fuzzy betweenness relation $\mathfrak{B}^{M}$ can not be the following
$
\mathfrak{B}^{M}(x, y, z)
=\bigwedge_{t>0}\left(M(x, z, t)\rightarrow
\left(M(x, y, s)\wedge M(y,z,r)\right)\right).
$
Because it fails to meet condition (FB2).

\noindent (2)
In subsection 4.3, we will show that $\mathfrak{B}^{M}$ not only satisfies the transitivity (FBR4) and (FBR5), but also satisfies all other four-point and five-point types of transitivity properties.
Meanwhile, there has an another method for proving (FBR4), which can be found in Theorem \ref{The relationship3 between fuzzy betweenness relations}.

\noindent (3) 
The $t$-norm is chosen to be the minimum one, because it was used in the proof of condition (FBR3) and (FBR4). 
Also, this leads us to conclude that 
$\mathfrak{B}^{M}$ is an fuzzy betweenness relation that satisfies more transitivity properties. 
\end{remark}


\begin{example}
Let $(X, d)$ be a metric space and
let $M_{d}$ be the standard KM-fuzzy metric 
under any  $t$-norm.
Then the fuzzy betweeenness relation 
$\mathfrak{B}^{M_{d}}: X^{3}\rightarrow [0, 1]$ 
is that 
$$\mathfrak{B}^{M_{d}}(x, y, z)
=\bigwedge_{t>0}\left(
\frac{t}{t+d(x, z)}
\rightarrow
\left(\bigvee_{s+r=t}\frac{s}{s+d(x, y)}\wedge \frac{r}{r+d(y, z)}\right)\right).$$
\end{example}


\subsection{A fuzzy betweenness relation $\mathfrak{B}^{\mathcal{D}^{M}}$ induced by its corresponding nest of metrics $\mathcal{D}^{M}$}

In this part, we shall give an another method to construct a fuzzy betweenness relation.
Since there is a one-to-one correspondence between a KM-fuzzy metric and its nests of metrics (under the $t$-norm $*=\wedge$, or more general case that $*$ is strictly monotone), we shall get the following important construction from another perspective.

\begin{theorem}
\label{a fuzzy betweenness relation2 induced by a KM-fuzzy metric}
Let $(X, M, \wedge)$ be a KM-fuzzy metric space and let
$\mathcal{D}^{M}=\{d_{a}^{M}\}_{a\in (0, 1)}$
be the corresponding nest of metrics.
Define a fuzzy ternary relation $\mathfrak{B}^{\mathcal{D}^{M}}: X^{3}\rightarrow [0, 1]$ by
 \begin{eqnarray*}
\mathfrak{B}^{\mathcal{D}^{M}}(x, y, z)
&=&\bigvee\left\{a\in (0, 1) \mid
d_{a}^{M}(x, z)=d_{a}^{M}(x, y)+d_{a}^{M}(y, z)\right\}\\
&=&\bigvee\left\{a\in (0, 1) \mid
d_{a}^{M}(x, z)\geq d_{a}^{M}(x, y)+d_{a}^{M}(y, z)\right\}.
\end{eqnarray*}
Then 
$\mathfrak{B}^{\mathcal{D}^{M}}$ is a fuzzy betweenness relation.

%
\end{theorem}

\begin{proof}
\noindent We need to check that $\mathfrak{B}^{\mathcal{D}^{M}}$ satisfies 
(FBR1)-(FBR5). For any $x, y, z\in X$,

\noindent (FBR1) $\mathfrak{B}^{\mathcal{D}^{M}}(x, y, z)=\mathfrak{B}^{\mathcal{D}^{M}}(z, y, x)$, obviously.

\noindent (FBR2) $\mathfrak{B}^{\mathcal{D}^{M}}(x, y, y)=\bigvee_{a\in (0,1)}a=1$.

\noindent (FBR3) 
Suppose that $\mathfrak{B}^{\mathcal{D}^{M}}(x, y, z)\wedge \mathfrak{B}^{\mathcal{D}^{M}}(x, z, y)=r\neq 0$.
Take any $\alpha<r$,
then $r\leq \mathfrak{B}^{\mathcal{D}^{M}}(x, y, z)$ and $r\leq \mathfrak{B}^{\mathcal{D}^{M}}(x, y, z)$. 
So there exist $a_{0}\in (0,1)$
and $b_{0}\in (0,1)$ with 
$d_{a_{0}}^{M}(x, z)=d_{a_{0}}^{M}(x, y)+d_{a_{0}}^{M}(y, z)$ and
$d_{b_{0}}^{M}(x, y)=d_{b_{0}}^{M}(x, z)+d_{b_{0}}^{M}(z, y)$ such that $\alpha<a_{0}$ and $\alpha<b_{0}$.
From Theorem \ref{theorem1 of betweenness relations in KM-fuzzy metric spaces}, we get 
$d_{\alpha}^{M}(x, z)=d_{\alpha}^{M}(x, y)+d_{\alpha}^{M}(y, z)$ and $d_{\alpha}^{M}(x, y)=d_{\alpha}^{M}(x, z)+d_{\alpha}^{M}(z, y)$.
Then $d_{\alpha}^{M}(y, z)+d_{\alpha}^{M}(z, y)=0$. This shows $d_{\alpha}^{M}(y, z)=0$, i.e.,  $y=z$.

\noindent (FBR4) and (FBR5)
Take any $\alpha <\mathfrak{B}^{\mathcal{D}^{M}}(o, x, y)\wedge \mathfrak{B}^{\mathcal{D}^{M}}(o, y, z)$.
Then 
$\alpha <\mathfrak{B}^{\mathcal{D}^{M}}(o, x, y)$ and 
$\alpha <\mathfrak{B}^{\mathcal{D}^{M}}(o, y, z)$, 
So there exist  $a_{0}\in (0,1)$ and $b_{0}\in (0,1)$ with 
$d_{a_{0}}^{M}(o, y)=d_{a_{0}}^{M}(o, x)+d_{a_{0}}^{M}(x, y)$ and
$d_{b_{0}}^{M}(o, z)=d_{b_{0}}^{M}(o, y)+d_{b_{0}}^{M}(y, z)$ such that $\alpha<a_{0}$ and $\alpha<b_{0}$.
From Theorem \ref{theorem1 of betweenness relations in KM-fuzzy metric spaces}, we get 
$d_{\alpha}^{M}(o, y)=d_{\alpha}^{M}(o, x)+d_{\alpha}^{M}(x, y)$ and
$d_{\alpha}^{M}(o, z)=d_{\alpha}^{M}(o, y)+d_{\alpha}^{M}(y, z)$.
Since $B_{d_{a}^{M}}=\{
(x, y, z)\in X^{3} \mid 
d_{a}^{M}(x, z)=d_{a}^{M}(x, y)+d_{a}^{M}(y, z)\}$ is a betweenness relation and it satisfies the transitivity (P2) and (P3),
it follows that
$d_{\alpha}^{M}(o, z)=d_{\alpha}^{M}(o, x)+d_{\alpha}^{M}(x, z)$ and
$d_{\alpha}^{M}(x, z)=d_{\alpha}^{M}(x, y)+d_{\alpha}^{M}(y, z)$.
Hence $\alpha\leq \mathfrak{B}^{\mathcal{D}^{M}}(o, x, z)$
and $\alpha\leq \mathfrak{B}^{\mathcal{D}^{M}}(x, y, z)$.
From the arbitrariness of $\alpha$, $\mathfrak{B}^{\mathcal{D}^{M}}(o, x, y)\wedge \mathfrak{B}^{\mathcal{D}^{M}}(o, y, z)\leq \mathfrak{B}^{\mathcal{D}^{M}}(o, x, z)$ and $\mathfrak{B}^{\mathcal{D}^{M}}(o, x, y)\wedge \mathfrak{B}^{\mathcal{D}^{M}}(o, y, z)\leq \mathfrak{B}^{\mathcal{D}^{M}}(x, y, z)$.
Therefore (FBR4) and (FBR5) both hold.
\end{proof}

Next, we will present the equivalent characterization of 
$\mathfrak{B}^{\mathcal{D}^{M}}$.

\begin{theorem}
\label{characterization1 of a fuzzy betweenness relation induced by a KM-fuzzy metric}
Let $(X, M, \wedge)$ be a KM-fuzzy metric space. 
Then 
for any $a\in (0,1)$, 
\begin{eqnarray*}
a\leq\mathfrak{B}^{\mathcal{D}^{M}}(x, y, z)
&\Leftrightarrow &
\forall b<a, ~d_{b}^{M}(x, z)=d_{b}^{M}(x, y)+d_{b}^{M}(y, z)\\
&\Leftrightarrow &
\forall b<a, ~d_{b}^{M}(x, z)\geq d_{b}^{M}(x, y)+d_{b}^{M}(y, z).
\end{eqnarray*}
\end{theorem}

\begin{proof}
On one hand, suppose that $d_{b}^{M}(x, z)\geq d_{b}^{M}(x, y)+d_{b}^{M}(y, z)$ for any $b<a$. 
By Theorem \ref{a fuzzy betweenness relation2 induced by a KM-fuzzy metric}, we have
$\mathfrak{B}^{\mathcal{D}^{M}}(x, y, z)\geq b$. Then
$\mathfrak{B}^{\mathcal{D}^{M}}(x, y, z)\geq \bigvee_{b<a}b=a$.
On the other hand, suppose that $a\leq\mathfrak{B}^{\mathcal{D}^{M}}(x, y, z)$.
For any $b<a$, $b<\mathfrak{B}^{\mathcal{D}^{M}}(x, y, z)
=\bigvee\{a\in (0, 1) \mid
d_{a}^{M}(x, z)\geq d_{a}^{M}(x, y)+d_{a}^{M}(y, z)\} $.
Then there exist some $a_{0}\in (0, 1)$
with $d_{a_{0}}^{M}(x, z)\geq d_{a_{0}}^{M}(x, y)+d_{a_{0}}^{M}(y, z)$
such that $b<a_{0}$.
By Theorem \ref{theorem1 of betweenness relations in KM-fuzzy metric spaces}, we get
$d_{b}^{M}(x, z)\geq d_{b}^{M}(x, y)+d_{b}^{M}(y, z)$.
\end{proof}

If we want to get the conclusion that 
$a\leq\mathfrak{B}^{\mathcal{D}^{M}}(x, y, z)$ if and only if 
$d_{a}^{M}(x, z)=d_{a}^{M}(x, y)+d_{a}^{M}(y, z)$,
then we need to add an additional condition that $M(x, y, t)$ is strictly monotonically increasing with respect to $t$. That is, 

\begin{theorem}
\label{characterization2 of a fuzzy betweenness relation induced by a KM-fuzzy metric}
Let $(X, M, \wedge)$ be a KM-fuzzy metric space. 
If $M(x,y,-): [0, \infty)\rightarrow [0, 1]$ 
is strictly increasing,
then 
for any $a\in (0,1)$, 
$$a\leq\mathfrak{B}^{\mathcal{D}^{M}}(x, y, z)
\Leftrightarrow 
d_{a}^{M}(x, z)=d_{a}^{M}(x, y)+d_{a}^{M}(y, z).$$
\end{theorem}

\begin{proof}
(1) Firstly, we prove the conclusion
$\bigvee_{b<a}d_{b}^{M}=d_{a}^{M}$.
From Theorem \ref{theorem of a nest of metric induced by a KM-fuzzy metric}, 
$d_{a}^{M}$ is monotonically increasing with respect to $a$. Then
$\bigvee_{b<a}d_{b}^{M}\leq d_{a}^{M}$.
It suffices to prove that 
$ d_{a}^{M}(x, y)\leq \bigvee_{b<a}d_{b}^{M}(x, y)$.
Take any $r>0$ with $r<d_{a}^{M}(x, y)=\bigvee\{
t\in [0, \infty) \mid M(x, y, t)\leq a \}$, then there exists some $t_{0}$ with $M(x, y, t_{0})\leq a $ such that $r<t_{0}$.
Since $M(x,y,-)$ is strictly increasing,
it follows that 
$M(x, y, r)<M(x, y, t_{0})\leq a$.
So $M(x, y, r)< a$.
This implies that there exists some $b<a$ such that
$M(x, y, r)< b<a$.
Then $M(x, y, r)\leq b$.
By Theorem \ref{theorem of a nest of metric induced by a KM-fuzzy metric}, we get 
$r\leq d_{b}^{M}(x, y)\leq \bigvee_{b<a}d_{b}^{M}(x, y)$.
From the arbitrariness of $r$, 
$ d_{a}^{M}(x, y)\leq\bigvee_{b<a}d_{b}^{M}(x, y)$.  
Therefore $\bigvee_{b<a}d_{b}^{M}=d_{a}^{M}$.  

\noindent (2)
Secondly, we prove that $a\leq\mathfrak{B}^{\mathcal{D}^{M}}(x, y, z)$
if and only if
$d_{a}^{M}(x, z)=d_{a}^{M}(x, y)+d_{a}^{M}(y, z).$
For any $a\in (0,1)$, it easy to see that
$d_{a}^{M}(x, z)=d_{a}^{M}(x, y)+d_{a}^{M}(y, z)$ implies
$a\leq\mathfrak{B}^{\mathcal{D}^{M}}(x, y, z).$
On the other hand, suppose that  $a\leq \mathfrak{B}^{\mathcal{D}^{M}}(x, y, z)$.
Take any $b<a\leq \mathfrak{B}^{\mathcal{D}^{M}}(x, y, z)$. 
Then there exists some $a_{0}\in (0,1)$ 
with $d_{a_{0}}^{M}(x, z)=d_{a_{0}}^{M}(x, y)+d_{a_{0}}^{M}(y, z)$
such that $b<a_{0}$. 
From Theorem \ref{theorem1 of betweenness relations in KM-fuzzy metric spaces}, we get $d_{b}^{M}(x, z)=d_{b}^{M}(x, y)+d_{b}^{M}(y, z)$. Then
$\bigvee_{b<a}d_{b}^{M}(x, z)=\bigvee_{b<a}\left(
d_{b}^{M}(x, y)+d_{b}^{M}(y, z)\right)$.
By (1), it follows from
$d_{a}^{M}(x, z)=\bigvee_{b<a}d_{b}^{M}(x, z)$ that $\bigvee_{b<a}\left(d_{b}^{M}(x, y)+d_{b}^{M}(y, z)\right)
=\bigvee_{b<a}\left(d_{b}^{M}(x, y)\right)+\bigvee_{b<a}\left(d_{b}^{M}(y, z)\right)=d_{a}^{M}(x, y)+d_{a}^{M}(y, z)$.
Hence $d_{a}^{M}(x, z)=d_{a}^{M}(x, y)+d_{a}^{M}(y, z)$.
\end{proof}

Finally, we will give an another equivalent characterization of the fuzzy betweenness relation $\mathfrak{B}^{\mathcal{D}^{M}}$ from the perspective of a family of betweenness relations.

\begin{theorem}
\label{The relationship2 between fuzzy betweenness relations}
Let $(X, M, \wedge)$ be a KM-fuzzy metric space. 
Then for any $a\in (0, 1)$,

\noindent {\rm (1)}
$$a\leq\mathfrak{B}^{\mathcal{D}^{M}}(x, y, z)
\Leftrightarrow 
\forall b<a, ~(x, y, z)\in B_{d_{b}^{M}}
\Leftrightarrow 
~(x, y, z)\in \bigcap_{b<a}B_{d_{b}^{M}}.$$

\noindent {\rm (2)}
$$\mathfrak{B}^{\mathcal{D}^{M}}(x, y, z)\
=\bigvee \{a\in (0,1) \mid (x, y, z) \in \bigcap_{b<a}B_{d_{b}^{M}}\},$$
where $B_{d_{b}^{M}}=\{
(x, y, z)\in X^{3} \mid 
d_{b}^{M}(x, z)=d_{b}^{M}(x, y)+d_{b}^{M}(y, z)\}$.
\end{theorem}

\begin{proof}
\noindent {\rm (1)}
For any $a\in (0,1)$, it follows from 
Theorem \ref{characterization1 of a fuzzy betweenness relation induced by a KM-fuzzy metric} that
\begin{eqnarray*}
a\leq\mathfrak{B}^{\mathcal{D}^{M}}(x, y, z)
&\Leftrightarrow &
\forall b<a, ~d_{b}^{M}(x, z)=d_{b}^{M}(x, y)+d_{b}^{M}(y, z)\\
&\Leftrightarrow &
\forall b<a, ~d_{b}^{M}(x, z)\geq d_{b}^{M}(x, y)+d_{b}^{M}(y, z)\\
&\Leftrightarrow &
\forall b<a, ~(x, y, z)\in B_{d_{b}^{M}}
\Leftrightarrow 
~(x, y, z)\in \bigcap_{b<a}B_{d_{b}^{M}}.
\end{eqnarray*}

\noindent {\rm (2)}
By Theorem \ref {a fuzzy betweenness relation2 induced by a KM-fuzzy metric} and (1), it is easily to seen that 
$$\mathfrak{B}^{\mathcal{D}^{M}}(x, y, z)
=\bigvee\{a\in (0, 1) \mid
d_{a}^{M}(x, z)\geq d_{a}^{M}(x, y)+d_{a}^{M}(y, z)\}
=\bigvee \{a\in (0,1) \mid (x, y, z) \in \bigcap_{b<a}B_{d_{b}^{M}}\}.$$
\end{proof}

\subsection{The relationship between fuzzy betweenness relations $\mathfrak{B}^{M}$ and $\mathfrak{B}^{\mathcal{D}^{M}}$}

In this part, we will discuss the connections between two types of fuzzy betweenness relations induced by a KM-fuzzy metric.
One of them is that we directly construct
a fuzzy betweenness relations by the KM-fuzzy metric $M$, that is, 
$$
\mathfrak{B}^{M}(x, y, z)
=\bigwedge_{t>0}\left(M(x, z, t)\rightarrow
\left(\bigvee_{s+r=t}M(x, y, s)\wedge M(y,z,r)\right)\right).
$$
The other one is constructed by its corresponding nest of metrics $\mathcal{D}^{M}=\{d_{a}^{M}\}_{a\in (0, 1)}$, that is, 
$$\mathfrak{B}^{\mathcal{D}^{M}}(x, y, z)
=\bigvee\left\{a\in (0, 1) \mid
d_{a}^{M}(x, z)\geq d_{a}^{M}(x, y)+d_{a}^{M}(y, z)\right\}.$$

\begin{theorem}
\label{The relationship1 between fuzzy betweenness relations}
Let $(X, M, \wedge)$ be a KM-fuzzy metric space and let
$\mathcal{D}^{M}=\{d_{a}^{M}\}_{a\in (0, 1)}$
be the corresponding nest of metrics.
Then 
$\mathfrak{B}^{M}=\mathfrak{B}^{\mathcal{D}^{M}}$.
\end{theorem}

\begin{proof} 
\noindent (1) 
We firstly prove that 
for any $a\in (0, 1)$ and for all $x, y, z\in X$,
$$\forall t>0, a\wedge M(x, z, t)\leq \bigvee_{s+r=t}M(x, y, s)\wedge M(y,z,r)\Leftrightarrow \forall b<a, d_{b}^{M}(x, z)\geq d_{b}^{M}(x, y)+d_{b}^{M}(y, z).$$

\noindent i) 
If $\forall t>0,~a\wedge M(x, z, t)\leq \bigvee_{s+r=t}M(x, y, s)\wedge M(y,z,r)$ holds. For any $b<a$.
take any $t=s+r>0$ with $d_{b}^{M}(x, y)+d_{b}^{M}(y, z)\geq t$. 
Then $d_{b}^{M}(x, y)\geq s$ and $d_{b}^{M}(y, z)\geq r$.
By Theorem \ref{theorem of a nest of metric induced by a KM-fuzzy metric}, we know  $M(x, y, s)\leq b$ and $M(y, z, r)\leq b$.
Then $ \bigvee_{s+r=t}M(x, y, s)\wedge M(y,z,r)\leq b$.
So $ a\wedge M(x, z, t)\leq b$.
It follows from $b<a$ that $ M(x, z, t)\leq b$.
From Theorem \ref{theorem of a nest of metric induced by a KM-fuzzy metric}, $d_{b}^{M}(x, z)\geq t$.
By the arbitrariness of $t$, we obtain
$d_{b}^{M}(x, z)\geq d_{b}^{M}(x, y)+d_{b}^{M}(y, z)$.

\noindent ii)
If $\forall b<a,~d_{b}^{M}(x, z)\geq d_{b}^{M}(x, y)+d_{b}^{M}(y, z)$ holds.
Take any $\beta>0$ with $\beta<a\wedge M(x, z, t)$.
Then $\beta<a$ and $\beta<M(x, z, t)$.
By Theorem \ref{theorem of a nest of metric induced by a KM-fuzzy metric}, we know $d_{\beta}^{M}(x, z)<t$. 
Since $d_{\beta}^{M}(x, z)\geq d_{\beta}^{M}(x, y)+d_{\beta}^{M}(y, z)$, 
it follows that $d_{\beta}^{M}(x, y)<s$ and $d_{\beta}^{M}(y, z)<r$ for any $t=s+r$.
From Theorem \ref{theorem of a nest of metric induced by a KM-fuzzy metric}, we get 
$\beta<M(x, y, s)$ and $\beta<M(y, z, r)$.
Then $\beta<M(x, y, s)\wedge M(y, z, r)$.
So $\beta\leq \bigvee_{s+r=t}M(x, y, s)\wedge M(y,z,r)$.
By the arbitrariness of $\beta$, 
$a\wedge M(x, z, t)\leq \bigvee_{s+r=t}M(x, y, s)\wedge M(y,z,r)$.

\noindent (2) By Theorem \ref{a fuzzy betweenness relation1 induced by a KM-fuzzy metric} and Theorem \ref{characterization1 of a fuzzy betweenness relation induced by a KM-fuzzy metric}, 
\begin{eqnarray*}
a\leq \mathfrak{B}^{M}(x, y, z)
&\Leftrightarrow &
\forall t>0, a\wedge M(x, z, t)\leq \bigvee_{s+r=t}M(x, y, s)\wedge M(y,z,r).\\
&\Leftrightarrow & \forall b<a,
d_{b}^{M}(x, z)\geq d_{b}^{M}(x, y)+d_{b}^{M}(y, z)
\Leftrightarrow  a\leq\mathfrak{B}^{\mathcal{D}^{M}}(x, y, z).
\end{eqnarray*}
Therefore $\mathfrak{B}^{M}(x, y, z)=\mathfrak{B}^{\mathcal{D}^{M}}(x, y, z)$, i.e., $\mathfrak{B}^{M}=\mathfrak{B}^{\mathcal{D}^{M}}$.
\end{proof}

Finally, we shall discuss the two type of fuzzy betweenness relations whether satisfy the other fuzzy transitivity properties.
By Theorem  \ref{The relationship1 between fuzzy betweenness relations}
and Theorem \ref{The relationship2 between fuzzy betweenness relations}, we get the following conclusions.

\begin{theorem}
\label{The relationship3 between fuzzy betweenness relations}
Let $(X, M, \wedge)$ be a KM-fuzzy metric space and let
$\mathcal{D}^{M}=\{d_{a}^{M}\}_{a\in (0, 1)}$
be the corresponding nest of metrics by $M$.
Then 
$\mathfrak{B}^{M}$ and $\mathfrak{B}^{\mathcal{D}^{M}}$ both satisfy {\rm (FP1)-(FP8)} and {\rm (FT1)-(FT6)}.

\noindent 
{\rm (FP1)} $T(x, s, t)*T(s, t, y)\leq T(x, s, y)$ 
for all $x, y, s, t\in X$;

\noindent 
{\rm (FP2)} $T(x, s, t)*T(s, y, t)\leq T(x, s, y)$ 
for all $x, y, s, t\in X$;

\noindent 
{\rm (FP3)} $T(x, s, t)*T(s, y, t)\leq T(x, y, t)$ 
for all $x, y, s, t\in X$;

\noindent 
{\rm (FP4)} $T(s, x, t)*T(s, y, t)\leq T(s, x, y)\vee T(s, y, x)$ 
for all $x, y, s, t\in X$;

\noindent 
{\rm (FP5)} $T(s, x, t)*T(s, y, t)\leq T(s, x, y)\vee T(y, x, t)$ 
for all $x, y, s, t\in X$;

\noindent 
{\rm (FP6)} $T(x, s, t)*T(y, s, t)\leq T(x, y, t)\vee T(y, x, t)$ 
for all $x, y, s, t\in X$;

\noindent 
{\rm (FP7)} $T (x, s, t)*T(y, s, t)\leq T(x, y, s)\vee T(y, x, s)$ 
for all $x, y, s, t\in X$;

\noindent 
{\rm (FP8)} $T (x, s, t)*T(y, s, t)\leq T(x, y, s)\vee T(y, x, t)$ 
for all $x, y, s, t\in X$.

\noindent 
{\rm (FT1)} $T(x, y, t)*T(s, t, z)\leq T(x, y, z)$ 
for all $x, y, s, t\in X$;

\noindent 
{\rm (FT2)} $T(x, y, t)*T(t, s, z)\leq T(x, y, z)$ 
for all $x, y, s, t\in X$;

\noindent 
{\rm (FT3)} $T(x, y, t)*T(t, z, s)\leq T(x, y, z)$ 
for all $x, y, s, t\in X$;

\noindent 
{\rm (FT4)} $T(s, x, t)*T(t, y, z)\leq T(x, y, z)$ 
for all $x, y, s, t\in X$;

\noindent 
{\rm (FT5)} $T(x, s, t)*T(t, y, z)\leq T(x, y, z)$ 
for all $x, y, s, t\in X$;

\noindent 
{\rm (FT6)} $T(x, s, t)*T(s, y, z)\leq T(x, y, z)$ 
for all $x, y, s, t\in X$.
\end{theorem}

\begin{proof} 
Firstly, we show $\bigcap_{b<a}B_{d_{b}^{M}}$ is 
a metric-betweenness relation.
From Theorem
\ref{theorem of a nest of metric induced by a KM-fuzzy metric}, we know $d_{a}^{M}$ is a metric for any $a\in (0,1)$.
Since Example \ref{example of betweenness relation of metric}, it follows that
$B_{d_{a}^{M}}=\{(x, y, z)\in X^{3} \mid 
d_{a}^{M}(x, z)=d_{a}^{M}(x, y)+d_{a}^{M}(y, z)\}$ is a metric-betweenness relation and it satisfies the transitivity properties of (P1)-(P8) and (T1)-(T6).
it is not difficult to check that $\bigcap_{b<a}B_{d_{b}^{M}}$ is also a metric-betweenness relation and also satisfies  (P1)-(P8) and (T1)-(T6).

Next, we prove 
$\mathfrak{B}^{\mathcal{D}^{M}}$ satisfies (FP1).
For any $x, y, s, t\in X$,
take any $a\in (0, 1)$ with $a\leq \mathfrak{B}^{\mathcal{D}^{M}}(x, s, t)
\wedge \mathfrak{B}^{\mathcal{D}^{M}}(s, t, y)$.
Then $a\leq \mathfrak{B}^{\mathcal{D}^{M}}(x, s, t)$ and $a\leq\mathfrak{B}^{\mathcal{D}^{M}}(s, t, y)$.
By Theorem  \ref{The relationship2 between fuzzy betweenness relations}, 
$(x, s, t)\in \bigcap_{b<a}B_{d_{b}^{M}}$ and 
$(s, t, y)\in \bigcap_{b<a}B_{d_{b}^{M}}$.
Since $\bigcap_{b<a}B_{d_{b}^{M}}$ is a metric-betweenness relation and satisfies the transitivity property (P1), we know 
$(s, t, y)\in \bigcap_{b<a}B_{d_{b}^{M}}$.
Then $a\leq \mathfrak{B}^{\mathcal{D}^{M}}(s, t, y)$.
From the arbitrariness of $a$, we obtain
$\mathfrak{B}^{\mathcal{D}^{M}}(x, s, t)
\wedge \mathfrak{B}^{\mathcal{D}^{M}}(s, t, y)\leq \mathfrak{B}^{\mathcal{D}^{M}}(x, s, y)$.
So (FP1) holds. 
The same process can be applied to the other cases and omitted here. 
Hence $\mathfrak{B}^{\mathcal{D}^{M}}$ satisfies (FP1)-(FP8) and (FT1)-(FT6). 
By Theorem \ref{The relationship1 between fuzzy betweenness relations}, we get 
$\mathfrak{B}^{M}=\mathfrak{B}^{\mathcal{D}^{M}}$.
Therefore $\mathfrak{B}^{M}$ and $\mathfrak{B}^{\mathcal{D}^{M}}$ both satisfy (FP1)-(FP8) and (FT1)-(FT6). 
\end{proof}

At the end of this paper, we can obtain the following summary diagram (see Figure \ref{summary of fuzzy betweenness relations induced by a KM-fuzzy metric}).

\begin{figure}[htbp]
\centering  
\includegraphics[height=7cm, width=15.5cm]
{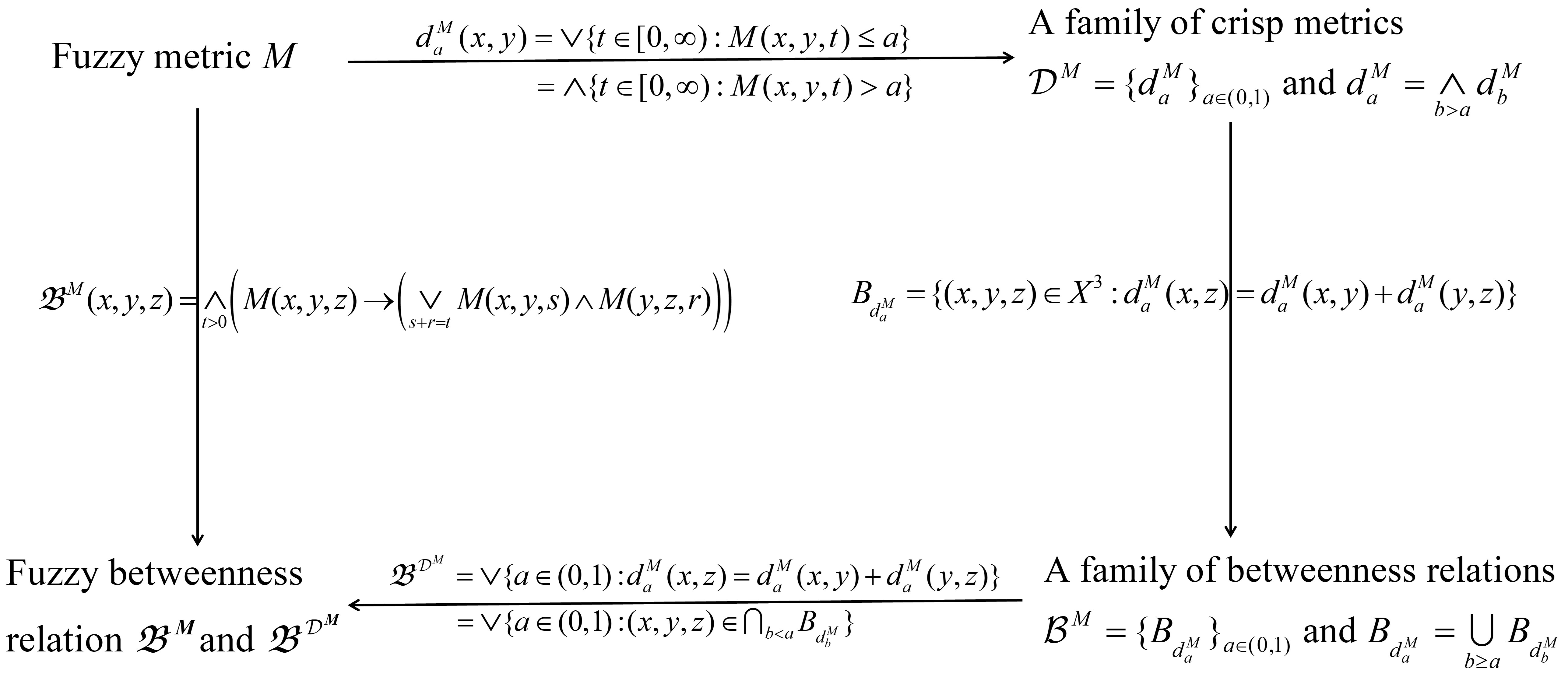}
\caption{Summary diagram of fuzzy betweenness relations induced by a KM-fuzzy metric.} 
\label{summary of fuzzy betweenness relations induced by a KM-fuzzy metric}
\end{figure}

\section{\bf Conclusions}

\noindent
{\bf Conflicts of interest:} 
The authors declare no conflict of interest.
\vskip2mm


\end{document}